\documentclass[sn-mathphys-num, 12pt]{sn-jnl}

\usepackage{amsmath,amssymb,amsfonts,amsthm,mathrsfs, latexsym, textcomp,amscd}%
\usepackage{multirow}
\usepackage[title]{appendix}
\usepackage{textcomp}
\usepackage{manyfoot}
\usepackage{booktabs}
\usepackage{algorithm}
\usepackage{algorithmicx}
\usepackage{algpseudocode}
\usepackage{listings}
\usepackage{natbib}
\usepackage[dvips]{epsfig}
\usepackage{graphics,verbatim}
\usepackage{upref, hyperref}
\usepackage{xcolor}
\usepackage{wrapfig}
\usepackage{enumitem} 
\usepackage[nameinlink]{cleveref}
\usepackage{orcidlink}

\theoremstyle{plain}
\newtheorem{theorem}{Theorem}[section]
\theoremstyle{plain}
\newtheorem{lemma}[theorem]{Lemma}
\newtheorem{proposition}[theorem]{Proposition}
\newtheorem{cor}[theorem]{Corollary}

\newtheorem{definition}{Definition}
\newtheorem{remark}{Remark}

\newcommand{\<}{\left\langle}
\renewcommand{\>}{\right\rangle}
\renewcommand{\(}{\left(}
\renewcommand{\)}{\right)}
\renewcommand{\[}{\left[}
\renewcommand{\]}{\right]}
\renewcommand{\{}{\left\lbrace }
\renewcommand{\}}{\right\rbrace }

\newcommand{\Be} {\begin{equation}}
\newcommand{\Ee} {\end{equation}}
\newcommand{\Nee} {\notag\end{equation}}
\newcommand{\bea} {\begin{eqnarray}}
\newcommand{\eea} {\end{eqnarray}}
\newcommand{\Bea} {\begin{eqnarray*}}
\newcommand{\Eea} {\end{eqnarray*}}

\newcommand{\abs}[1]{\left\vert#1\right\vert}

\newcommand{\norm}[1]{\left\Vert#1\right\Vert}

\newcommand{\De} {\Delta}
\newcommand{\la} {\lambda}
\newcommand{\rn}{\mathbb{R}^{N}}

\newcommand{\N}{\ensuremath{\mathbb{N}}}
\newcommand{\R}{\ensuremath{\mathbb{R}}}

\newcommand{\Bee}{\mathsf{B}}

\newcommand{\eps}{\varepsilon}
\newcommand{\dv}{\mathrm{~d} V_{\bn}}
\newcommand{\dx}{\mathrm{~d}x}
\newcommand{\dy}{\mathrm{~d}y}

\newcommand{\al} {\alpha}

\newcommand{\Ga} {\Gamma}

\newcommand{\Ro}{\varrho}

\newcommand{\f}{\frac}
\newcommand{\Sl}{\mathbb S}

\newcommand{\bn}{\mathbb{B}^{N}}

\begin{document}


\title[Sign changing solutions]{Existence and non-existence of nonradial solutions to an elliptic equation on hyperbolic space with critical and subcritical nonlinearities}

\author[1]{\fnm{Atanu} \sur{Manna} \orcidlink{0009-0001-3419-3440}}

\author*[2]{\fnm{Bhakti Bhusan} \sur{Manna} \orcidlink{0000-0002-8098-2782}}

\abstract{ 	
		In this article, we study the semi-linear elliptic problem:
		\begin{align*}
			-\Delta_{\mathbb{B}^N} u  \, - \,  \lambda  u = |u|^{p-1}u,  \quad  u \in H^{1}\(\mathbb{B}^N\), 
		\end{align*}
		where $N \geq 3$, $\lambda > \frac{N(N-2)}{4}$, and $1 < p \leq 2^*-1$. Here, $\mathbb{B}^N$ represents the Poincar\'e ball model of the hyperbolic space and $H^{1}(\mathbb{B}^N)$ denotes the Sobolev space on $\mathbb{B}^N$. In the critical case $p = 2^*-1$, with $\lambda < \frac{(N-1)^2}{4}$ and $N \geq 4$, we establish the existence and multiplicity of nonradial sign-changing solutions. Moreover, our result answers an open question: ``the existence of sign-changing solutions in the lower dimensions $4 \leq N \leq 6$". We also prove a partial non-existence result for a large class of symmetric solutions when $\lambda > \frac{(N-1)^2}{4}$, based on a quantitative orbit packing condition.
		
}



\footnotetext[2]{\textbf{*Corresponding author: Bhakti Bhusan Manna}, IIT Hyderabad, Department of Mathematics, Kandi, Sangareddy, Telangana, 502284, India, email: bbmanna@math.iith.ac.in}

\footnotetext[1]{\textbf{Atanu Manna}, IIT Hyderabad, Department of Mathematics, Kandi, Sangareddy, Telangana, 502284, India, email: mannaatanu98@gmail.com}

\maketitle


	\section{Introduction}
The objective of this article is twofold: first, we study the existence of nonradial sign-changing solutions for the critical case, and secondly, we establish symmetry-based non-existence results for spectral values. Let us consider the following semilinear elliptic equation, known as the Poincar\'e-Sobolev equation in the hyperbolic space $\bn$:
\begin{align*}
-\De_{\bn} u  \, - \,  \la  u &= |u|^{p-1}u,  \quad  u \in H^{1}(\bn), \qquad \label{(1)} \tag{$Eq_\la$}
\end{align*}
where $N \geq 3$, $\lambda > \frac{N(N-2)}{4}$, and $1 < p \leq 2^*-1$. $H^1\(\bn\)$ denotes the Sobolev space defined on the $N$-dimensional hyperbolic space and $\la_1 = \f{(N-1)^2}{4}$ is the bottom of the $L^2$-spectrum of $- \De_{\bn}$, the Laplace-Beltrami operator.\\

Elliptic equations involving the critical Sobolev exponent form one of the central themes in nonlinear analysis. The main difficulty comes from the fact that the Sobolev embedding $H_0^1\(\Omega\) \hookrightarrow L^{2^*}(\Omega)$ is continuous but not compact. As a consequence, variational sequences may lose mass by concentration, and the standard direct methods of the calculus of variations do not apply without additional arguments.\\

Let us consider the critical equation in the entire Euclidean space:
\begin{align*}
	- \Delta w = |w|^{2^*-2}w, \quad w \in D^{1,2}(\rn), \label{equation in R^N} \tag{$Eq^\infty$}
\end{align*}
whose positive finite-energy solutions are the Aubin–Talenti bubbles \cite{Aubin,Talenti}. In \cite{Brezis_Nirenberg}, Br\'ezis and Nirenberg showed that lower-order perturbations can recover compactness and yield positive solutions in bounded domains. This work generated an extensive literature on critical elliptic equations, including existence, multiplicity, non-existence, concentration, and asymptotic analysis. The concentration-compactness method of Lions and the global compactness decomposition of Struwe \cite{Struwe} have become fundamental tools for precisely describing how Palais–Smale sequences fail to be compact.\\

Topology and symmetry have also played a decisive role in critical problems. In \cite{Bahri_Coron}, Bahri and Coron proved that the topology of the domain can produce positive solutions of critical problems through refined variational and topological arguments. In another direction, Ding \cite{Ding} used the conformal equivalence between $\rn$ and the sphere $\mathbb{S}^N$, together with suitable symmetries, to obtain sign-changing solutions to \eqref{equation in R^N}. Ding’s method illustrates
an important principle: symmetry can force the change of signs of solutions and can also remove some compactness defects by restricting the variational problem to a more compact symmetric class. Compared to positive solutions, showing the existence of sign-changing entire solutions requires more subtle analysis. In \cite{Del Pino_Musso_Paccard_Pistoia}, del Pino, Musso, Pacard, and Pistoia used Lyapunov-Schmidt reduction to show the existence of sign-changing clusters of bubbles that solve \eqref{equation in R^N}. In recent work \cite{Clapp_Rios}, Clapp and Rios developed a symmetry approach by considering an equivariant subspace and established existence and multiplicity results for nonradial sign-changing solutions to \eqref{equation in R^N}.\\

Now, we focus on the critical case $p = 2^*-1$ of our equation \eqref{(1)}, which is in hyperbolic space. In \cite{Mancini_Sandeep}, Mancini and Sandeep established sharp existence and non-existence results for positive entire solutions to the critical Poincar\'e-Sobolev equation. They proved that every positive solution is radially symmetric about some point and is unique up to hyperbolic isometries. For $N \geq 4$, existence holds when $\frac{N(N-2)}{4} < \lambda < \lambda_1$, while nonexistence holds for $\la \leq \frac{N(N-2)}{4}$. Moreover, when $N=3$, there is no positive entire solution.\\

In \cite{Bhakta_Sandeep}, Bhakta and Sandeep proved the existence of a radial sign-changing solution to the critical Poincar\'e-Sobolev equation, if $N \geq 7$, $\frac{N(N-2)}{4} < \lambda < \lambda_1$, and non-existence if $\la \leq \f{N(N-2)}{4}$ or $N=3$. To show the existence, their approach includes firstly identifying an appropriate energy level for sign-changing solutions and then showing that it lies below the threshold at which bubbling can occur, using Palais-Smale decomposition in the spirit of Struwe. The dimensional restriction $N \geq 7$ is needed to obtain the required energy bound. Later in \cite{Ganguly_Sandeep}, Ganguly and Sandeep established a non-existence result for any nontrivial solution to the critical equation, when $N \geq 3$ and $\la \leq \frac{N(N-2)}{4}$. Also, they obtained infinitely many radial sign-changing solutions by subcritical approximation $p_n \to 2^*-1$. For fixed $n$, they took a sequence of radial sign-changing solutions with uniform $H^1$-bounds and then passed to the critical limit. They established a compactness lemma, which rules out the blow-up when $N \geq 7$, while passing to the critical exponent.\\

In our earlier articles, we established that for $1<p<2^*-1, \; \la < \la_1$, there exist infinitely many nonradial sign-changing solutions when $N = 4 \ \text{or} \ N \geq 6$, two such solutions for dimension $N = 5$, and one for dimension $N=3$. In \cite{Manna_Manna1}, the argument combines the construction of an equivariant subspace of $H^1(\bn)$, the analysis of the mountain pass level Palais-Smale sequence, the fountain theorem, and the compactness mechanisms adopted in the hyperbolic space. These particular kind of equivariant subspaces exist when $N\geq 4$. In \cite{Manna_Manna2}, the solutions constructed are different in nature from those in the previous article, thus improving the multiplicity results and the existence of nonradial sign-changing solutions in lower dimensions.\\

The preceding results leave us with two natural questions in the critical case of \eqref{(1)}. Firstly, while the existence of radial sign-changing solutions is known for dimensions $N \geq 7$, and non-existence for $N=3$, the existence of sign-changing solutions in lower dimensions was unknown. So the question is:\\

\begin{enumerate}
	\item[\textbf{Q1.}] \textit{Does the critical Poincar\'e-Sobolev equation \eqref{(1)} admit any sign-changing finite energy solution in the dimensions $4 \leq N \leq 6$?}
\end{enumerate}

\medskip

Secondly, all previously known existence results for sign-changing solutions in the critical case concern radial solutions. To the best of our knowledge, no existence result for nonradial sign-changing solutions was available in any dimension. So we wish to answer the following question:\\

\begin{enumerate}
	\item[\textbf{Q2.}] \textit{Beyond the radial class, does the critical Poincar\'e-Sobolev equation possess any nonradial finite energy solutions?}
\end{enumerate}

\medskip

We answer both questions affirmatively by establishing the existence of nonradial sign-changing solutions for dimensions $N \geq 4$, together with an explicit multiplicity depending on the dimension. We prove the following theorem:\\

\begin{theorem}\label{Existence theorem}
	Let $N \geq 4$, $ \frac{N(N-2)}{4} < \lambda < \frac{(N-1)^2}{4}, \; \text{and} \; p = 2^*-1$, then \eqref{(1)} admits at least $\lfloor \f{N}{4} \rfloor$ distinct nonradial sign-changing solutions.
\end{theorem}

\medskip

\medskip

When $\la > \la_1$, the equation \eqref{(1)} is known as Helmholtz type equation. In \cite{Mancini_Sandeep}, Mancini and Sandeep have shown that if $\la > \la_1$ and $p>1$, \eqref{(1)} does not possess any positive solution in $H^1\(\bn\)$. This can be proved by the fact that there exists $\eps > 0$ and a function $\psi \in C_c^\infty \(\bn\)$ such that
\begin{align*}
	\la - \eps > \la_1, \; \text{and} \; \la - \eps = \f{\int_{\bn} \abs{\nabla_{\bn} \psi}^2 \dv }{\int_{\bn} \psi^2 \dv}, \; \text{i.e.,} \; \int_{\bn} \abs{\nabla_{\bn} \psi}^2 \dv - \la \int_{\bn} \psi^2 \dv < 0.
\end{align*}
Now multiplying \eqref{(1)} by the test function $\f{\psi^2}{u}$, where $u$ is a positive solution to \eqref{(1)} and doing the integration by parts, we obtain
\begin{align*}
	\int_{\bn} \abs{\nabla_{\bn} \psi}^2 \dv - \la \int_{\bn} \psi^2 \dv = \int_{\bn} u^{p-1} \psi^2 \dv + \int_{\bn} \f{1}{u^2} \( \psi \abs{\nabla_{\bn}u} - u \abs{\nabla_{\bn} \psi}\)^2 \dv.
\end{align*}
Then the LHS is negative, but the RHS is positive. Therefore, there cannot be any positive solution to \eqref{(1)}.\\

Furthermore, for $\la > \la_1$ one can use a comparison argument to show that the positive solution must decay no faster than $\(\sinh d_{\bn}(x,0)\)^{-\frac{N-2}{2}}$, which is insufficient for square integrability, see the non-existence theorem in the log-perturbed Br\'ezis-Nirenberg problem on $\bn$, \cite{Ghosh_Joseph_Karmakar}. In \cite{Casteras_Mandel}, Casteras and Mandel established the nonexistence of radial solutions in $H^1\(\bn\)$, by studying the ODE corresponding to \eqref{(1)}. They proved that nontrivial radial solutions are oscillatory and their optimal decay rate is given by $\abs{u(r)} \approx \(\sinh r\)^{-\frac{N-1}{2}}$, as $r \to + \infty$. Consequently, such solutions do not belong to $L^2\(\bn\)$, and hence do not belong to $H^1\(\bn\)$; however they belong to $W^{2,s}(\bn)$, where $s>2$. With these results, we pose the following question:
\begin{align*}
	\textbf{Q3.} \; \textit{Do there exist nonradial solutions to \eqref{(1)}, when $\la > \la_1$ and $p>1$?} \hspace*{5cm}
\end{align*} 
In this article, we give a partial negative answer by showing the nonexistence of $H^1$-solutions in a large class of function spaces, which contains the radial Sobolev space $H_r^1\(\bn\)$, as well as an infinite-dimensional subspace consisting only of nonradial functions. We use the spectral theory of Schr\"odinger type operator. Firstly, we realize that \eqref{(1)} can be written as
\begin{align}
	\(- \De_{\bn} + V \) u = \la u, \quad \text{where} \; V = - |u|^{p-1}.  \label{Schrodinger operator form equation} \tag{$Eq'_\la$}
\end{align} 
Thus a finite energy solution to \eqref{(1)} is an $L^2$-eigenfunction of the Schr\"odinger type operator $- \De_{\bn} + V$. In \cite{BM}, Borthwick and Marzuola considered hyperbolic space $\bn$, with $N \geq 3$ and a real potential $V$. Then they proved
\begin{align}
	V \in L^\infty \(\bn\), \; V = o(\rho^{-1}) \; \text{as} \; \rho \to \infty, \; (-\De_{\bn} + V) \psi = \la \psi, \; \la > \la_1, \; \psi \in H^1\(\bn\) \Rightarrow \psi = 0. \label{Marzuola and Borthwick result}
\end{align}

Here $\(\rho, \theta\)$ is the geodesic polar coordinate in $\bn$. We establish the uniform decay estimate of the potential $V$, for various classes of functions in $H^1\(\bn\)$. The functions are mainly invariant under some subgroups of $O(N)$. The main goal is to establish a sufficient condition on the subgroups, so that invariant functions have enough decay to satisfy the hypothesis of \eqref{Marzuola and Borthwick result}. Let $\Ga$ be a closed subgroup of $O(N)$, then we define a maximal number of points that are separated by some fixed Euclidean distance and placed on the orbit on the unit sphere.\\

\begin{definition}
	Let $\Ga \subset_{\text{closed}} O(N)$. For $\theta \in \Sl^{N-1}$ and $\eta > 0$, define
	\begin{align*}
		M_\Ga (\theta, \eta) := \sup \{m \in \N \ : \ \exists \ g_1, \cdots, g_m \in \Ga \ \text{such that} \  |g_i \theta - g_j \theta| \geq \eta, \ \text{for} \ i \neq j\}.
	\end{align*}
\end{definition}

\vspace*{2mm}

Now, we introduce a \textbf{quantitative orbit packing} condition depending on a closed subgroup $\Ga$ of $O(N)$ and a positive real $d$. This condition will be called $(QOP)_{\Ga, d}$.\\

\begin{definition}\label{orbit packing number}
	Let $\Ga$ be a closed subgroup of $O(N)$. We say that $\Ga$ satisfies $(QOP)_{\Ga, d}$ for some $d \in (0, N-1]$, if there exist constants $C_0 >0, \ \eta_0 > 0$ such that
	\begin{align*}
		M_\Ga (\theta, \eta) \geq C_0 \eta^{-d},
	\end{align*}
	for every $\theta \in \Sl^{N-1}$ and every $0 < \eta < \eta_0$.
\end{definition}

\medskip

With this quantitative orbit packing condition enforced on the subgroups, we can prove the following non-existence theorem:\\

\begin{theorem}\label{Non-existence theorem for higher eigenvalues}
	Let $N \geq 3, \; \la > \la_1, \; 1<p\leq 2^*-1$ and let $u$ solve \eqref{(1)}. If $\abs{u}$ is $\Ga$-invariant, where $\Ga \subset O(N)$ satisfies the condition $(QOP)_{\Ga, d}$, then $u \equiv 0$ in $\bn$.
\end{theorem}

\medskip

The quantitative orbit packing condition also has a compactness consequence. For the block-radial functions, the growth of the group orbit forces the $L^q$-mass contained in a fixed geodesic ball $B_\zeta(x)$ to decrease as $d_{\bn}(0,x) \to \infty$. Moreover, the decay rate reflects the geometry and dimensions of the individual blocks. This decay yields compactness of the corresponding Sobolev embedding in the subcritical range. The details are in the Appendix (\Cref{Decay and compactness}).

\medskip

The paper is organized as follows.

\medskip

\begin{enumerate}
	\item[A.] \Cref{notations and preliminaries}: We set up the geometric and variational framework used throughout this article. We introduce the equivariant Sobolev spaces associated with closed subgroups of $O(N)$, together with some properties involving the orbit and fixed point set. Finally, we define a corresponding group representation on the Sobolev space and a group-averaging projection onto the equivariant subspace.
	
	\item [B.] \Cref{Limit problem}: We study the Euclidean limit problem associated with the critical equation and establish a strict comparison between the equivariant energy level of the hyperbolic problem and that of the Euclidean limit problem.
	
	\item[C.] \Cref{section for existence}: We prove that the equivariant variational level for the critical equation is attained. The compactness argument depends on the fixed point set of the group; accordingly, this section has been divided. We then provide explicit examples of groups that satisfy our assumptions. Applying the preceding existence results to these groups, we obtain the existence and multiplicity result for nonradial sign-changing solutions stated in \Cref{Existence theorem}.
	
	\item[D.] \Cref{non-existence section}: We turn to the spectral regime $\la > \la_1$. We show how the condition $(QOP)_{\Ga,d}$ ensures the pointwise decay estimates of symmetric solutions. Then we prove the non-existence result in \Cref{Non-existence theorem for higher eigenvalues}.
	
	\item[E.] \Cref{Appendix}: We present two auxiliary results. First, we prove the strong continuity of the group action. We then study block radial Sobolev functions and derive a quantitative decay estimate for their local $L^q$-mass. As a consequence, we get compactness of the corresponding equivariant Sobolev embedding in the subcritical case.
\end{enumerate}

\section{Notations and Preliminaries} \label{notations and preliminaries}
We start by recalling the Poincar\'e ball model of hyperbolic space, denoted by $\bn$, which is  the open unit ball $B(0,1)$ endowed with the Riemannian metric $g_{\bn}$ with $\(g_{\bn}\)_{ij} = \(\frac{2}{1 - |x|^2}\)^{2} \delta_{i j}$. The hyperbolic volume element, denoted as $\mathrm{~d} V_{\bn}$, can be expressed as $\mathrm{~d} V_{\bn} = \(\frac{2}{1 - |x|^2}\)^{N} \mathrm{~d} x$. The hyperbolic distance between $x,y \in \bn$ is given by
\begin{align}
	d_{\bn}(x,y) & := \cosh^{-1} \( 1 + \f{2|x-y|^2}{\(1-|x|^2\)\(1-|y|^2\)}\), \label{hyperbolic metric}
\end{align}
and we will denote $\rho(x) := d_{\bn}(0,x) = \log \(\f{1+|x|}{1-|x|}\)$. The geodesic ball and geodesic sphere in $\bn$ with center $x \in \bn$ and radius $r > 0$ are defined respectively by
\begin{align*}
	B_r(x) := \{ y \in \bn : d_{\bn}(x,y) < r\}, \ \ 
	S_r(x) := \{ y \in \bn : d_{\bn}(x,y) = r\}.
\end{align*}
If $x =0$, then the notations for geodesic ball and sphere are $B_r$ and $S_r$ respectively. From now on, we will denote the Euclidean balls $B(x,r)$ by $\Bee(x,r)$, and simply $\Bee$ for $B(0,1)$. Also, we note the relation $B_r = \Bee\(0, \tanh \f{r}{2}\)$.\\

The geodesic polar coordinate is given by $x = (\rho, \theta)$, where $\rho \in \R_+ \cup \{0\}$ and $\theta \in \Sl^{N-1}$. Then the hyperbolic metric can be rewritten as
\begin{align*}
	g_{\bn} = d\rho^2 + (\sinh \rho)^2 g_{\Sl^{N-1}},
\end{align*}
where $g_{\Sl^{N-1}}$ is the standard round metric on sphere. Furthermore, for the same $x \in \Bee$, the Euclidean polar coordinate will be denoted as $x = \(s, \theta\)$, where $s = \tanh \f{\rho}{2} \in [0,1)$.\\

\textbf{Hyperbolic translation:} For $b \in \bn$, we define the hyperbolic translation $\tau_b : \bn \rightarrow \bn$ as
\begin{align*}
\tau_b(x) &= \f{\(1-|b|^2\)x + \(|x|^2 + 2x \cdot b + 1\)b} {|b|^2|x|^2 + 2 x \cdot b + 1}.
\end{align*}
Note $\tau_b(0) = b, \, \forall\, b \in \bn$, and it acts as a translation along the line $\(- \f{b}{|b|}, \f{b}{|b|}\)$. To obtain further details on this topic, we recommend consulting the reference \cite{Ratcliffe}.

\subsection{Sobolev space on $\bn$ and a sharp Poincar\'e-Sobolev inequality:} 
Define
\begin{align*}
H^1(\bn) & := \{u \in L^2 (\bn) \;:\; \nabla_{\bn} u \in L^2\(\bn\)\},
\end{align*}
with the norm $\norm{u}_{H^1(\bn)} := \[\int_{\bn} \[u^2+\abs{\nabla_{\bn} u}^2\] \mathrm{~d} V_{\bn}\]^{\f{1}{2}}$. We also recall the sharp Poincar\'e-Sobolev inequality: For $N \geq 3$, $\la \leq \f{(N-1)^2}{4}$, and $p \in \left(1, \frac{N+2}{N-2} \right]$, there exists an optimal constant 
$S_{\la,p} > 0$ such that the following holds:
\begin{equation}
S_{\la,p} \left( \int_{\bn} |u|^{p + 1} \mathrm{~d} V_{\bn} \right)^{\frac{2}{p + 1}} \leq \int_{\bn} \left[|\nabla_{\bn} u|^{2}
- \la u^{2}\right] \mathrm{~d} V_{\bn}, \label{Poincare-Sobolev}
\end{equation}
for every $u \in H^1(\bn)$.\\

We also recall for $N \geq 3$, the classical Sobolev inequality: there exists $S>0$ (the best Sobolev constant) such that
\begin{equation}
	S \( \int_{\rn} |u|^{2^*} \mathrm{~d}x \)^{\f{2}{2^*}} \leq \int_{\rn} \abs{\nabla u}^2 \mathrm{~d}x, \quad \forall u \in C_c^\infty\(\rn\). \label{Sobolev}
\end{equation}
It is well known that $S_{\lambda, \; 2^*-1} < S$ if $\lambda > \f{N(N-2)}{4}$ and $S_{\lambda, \; 2^*-1} = S$ for $\lambda \leq \f{N(N-2)}{4}.$\\

 The bottom of the $L^2$-spectrum of $-\De_{\bn}$ is defined as
\begin{align*}
\la_1 \(\bn\) := \inf_{u \in H^1\(\bn\) \backslash \{0\}} \f{\int_{\bn} \abs{\nabla_{\bn} u}^2\mathrm{~d} V_{\bn} }{\int_{\bn} u^2 \mathrm{~d} V_{\bn}} = \f{(N-1)^2}{4}.
\end{align*}
Consequently, we have that for any $\la < \f{(N-1)^2}{4}$,
$ \norm{u}_\la := \[\int_{\bn} \[\abs{\nabla_{\bn} u}^2 - \la u^2\] \mathrm{~d} V_{\bn}\]^{\f{1}{2}}$
is an equivalent norm on $H^1\(\bn\)$, and let us denote the corresponding inner product by  $\< \cdot, \cdot \>_\la$. \\

\begin{lemma} \label{change of variable}
	For any two elements $u$ and $v$ belonging to the space $H^1(\bn)$, the following holds:
	\begin{align*}
		(i) \;& \int_{\bn} \< \nabla_{\bn} \(u \circ \tau_b\)\,,\, \nabla_{\bn}\(v \circ \tau_b\) \>_{\bn} \mathrm{~d} V_{\bn} = \int_{\bn}  \<\(\nabla_{\bn} u\) \circ \tau_b \,,\, \(\nabla_{\bn} v\) \circ \tau_b\>_{\bn} \mathrm{~d} V_{\bn}.\\
		(ii)\; & \int_{\bn} \(u \circ \tau_b\)(x) v(x) \mathrm{~d} V_{\bn} = \int_{\bn} u(x) \(v \circ \tau_{-b}\)(x) \mathrm{~d} V_{\bn}.
	\end{align*}
	Additionally, for any open subset $\Omega$ of $\bn$ and a measurable function $u$, we obtain
	\begin{align*}
		\int_{\Omega} \abs{u \circ \tau_b}^p \mathrm{~d} V_{\bn} & = \int_{\tau_b\(\Omega\)} \abs{u}^p \mathrm{~d} V_{\bn}, \quad 1 \leq p < \infty,
	\end{align*}
	assuming that the integrals have finite values.
\end{lemma}

\medskip

\subsection{Conformal change of metric:} 
For a conformal diffeomorphism $f: (M_1, g_1) \to (M_2, g_2)$ with $f^*g_2 = \Phi^{\frac{4}{N-2}}g_1$, the Conformal Laplacian $L_g := -\Delta_g + \frac{N-2}{4(N-1)}S_g$ ($S_g$ is the scalar curvature) satisfies the covariance property:
\begin{align*}
	L_{g_1}(\Phi(v \circ f)) = \Phi^{\frac{N+2}{N-2}} (L_{g_2} v) \circ f.
\end{align*}
In particular, if $v$ solves $L_{g_2}v = |v|^{2^*-2}v$ on $M_2$ if and only if $u = \Phi(v \circ f)$ solves $L_{g_1}u = |u|^{2^*-2}u$ on $M_1$; moreover the $L^{2^*}$-norm is preserved: $\|u\|_{L^{2^*}(M_1)} = \|v\|_{L^{2^*}(M_2)}$. In particular, for $p = 2^*-1$, and $\Phi = h^{\frac{N-2}{2}}$, where $h(x) = \frac{2}{1-|x|^2}$. the equation \eqref{(1)} transformed to 
\begin{equation}
	- \Delta v - \tilde{\lambda} h^2 v = |v|^{2^*-2}v, \quad v \in H_0^1(\Bee), \label{conformal equation} \tag{$CEq_\la$}
\end{equation}
where $\tilde{\lambda} = \lambda - \frac{N(N-2)}{4}$.

When $\f{N(N-2)}{4} < \la < \f{(N-1)^2}{4}$, or equivalently, $0 < \tilde{\la} < \frac{1}{4}$, the quadratic term: 
\Be Q(v)^{\f{1}{2}} := \(\int_{\Bee} \abs{\nabla v}^2 \dx - \tilde{\la} \int_{\Bee} h^2 v^2 \dx\)^{\f{1}{2}}\Nee
 is an equivalent norm on $H_0^1(\Bee)$. This follows from the conformal version of the Poincar\'e-Sobolev inequality, which we get by substituting the conformal factor $u = \(\frac{2}{1-|x|^2}\)^{- \frac{N-2}{2}}v$ in \eqref{Poincare-Sobolev};
\begin{align*}
	S_{\la,p} \(\int_{\Bee} |v|^{p+1} \(\frac{2}{1-|x|^2}\)^{\al (p)}  \dx \)^{\frac{2}{p+1}}  \leq \int_{\Bee} \[ |\nabla v|^2 - \tilde{\la} \(\frac{2}{1-|x|^2}\)^2 v^2 \dx\], \quad \forall v \in H_0^1(\Bee),
\end{align*} 
where $p \in \Big{(} 1 , \frac{N+2}{N-2} \Big{]}$ and $\al(p) = N - (p+1) \frac{N-2}{2}.$\\

\subsection{Equivariant functions }

\vspace*{0.3cm}
We shall now introduce the notion of group action and form a symmetric space of functions consisting only nonradial sign-changing functions. Let $\Ga$ be a closed subgroup of $O(N)$, then $\Ga$ acts on $\rn$ through the action
\begin{align*}
	\Ro : \Ga \times \rn \to \rn \; \text{such that} \; \(g,x\) \mapsto \Ro \(g,x\) = gx.
\end{align*}
For this action, we define
\begin{enumerate}
	\item $\Ga (x) :=\{gx\;:\; g \in \Ga \}$, for $x \in \rn$, to be the orbit of $x$.
	\item $\Ga_x := \{g \in \Ga : gx=x\}$ is the isotropy subgroup of $x \in \rn$.
	\item $\mathrm{Fix}_{\rn}(\Ga) := \{x \in \rn \;:\; gx=x, \forall g \in \Ga \}$ is the fixed point set under the group action.
\end{enumerate}

\medskip

Since orthogonal matrices preserve the Euclidean norm, it follows from equation (\ref{hyperbolic metric}) that the hyperbolic metric $d_{\bn}$ is also preserved by the group action of $O(N)$. Therefore, a closed subgroup $\Gamma$ of $O(N)$ produces the same action on $\bn$. Similarly, we can define the orbit and isotropy subgroup for a point $x \in \bn$; they will be the same as in Euclidean space. We denote the fixed point set in $\bn$ to be $\mathrm{Fix}_{\bn}(\Ga) := \{x \in \bn : gx = x, \; \forall g \in \Ga\}$.

\medskip

Let $\Omega$ be a $\Ga$-invariant domain in $\rn$ or $\bn$, not necessarily bounded; i.e., $\Ga (x) \subset \Omega, \; \forall x \in \Omega$. Now we define equivariant functions:\\

\begin{definition}
	Consider a continuous homomorphism $\phi$ from the group $\Ga$ onto the group $\mathbb{Z}_2$, where $\mathbb{Z}_2$ is the group consisting of the elements 1 and -1. A function $u: \Omega \to \R$ is considered $\phi$-equivariant if it satisfies the equation
	\begin{align}
		u(g x)=\phi(g)u(x),\qquad \forall g \in \Ga\ , x \in \Omega. \label{equivariant map}
	\end{align}
\end{definition}
We observe that $u$ is $\Lambda := \mathrm{ker}(\phi)$ invariant, and if $\phi$ is onto, then every nonzero $\phi$-equivariant function is nonradial and exhibits sign-changing behavior. This is true even for hyperbolic space $\bn$, as we know the geodesic sphere $S_r$ coincides with the Euclidean sphere $S\(0, \tanh \f{r}{2}\)$, for any $r \in (0,\infty)$.\\

Throughout this article, we make two assumptions about the group $\Ga$ and the onto homomorphism $\phi$ as:
\begin{equation}
	\exists \,\,\,\zeta \in \Bee \subset \rn \text{ \,  such \ that \,\, } \Ga_\zeta \subset \mathrm{ker} \,\phi. \label{(A_1)} \tag{$\mathbf{A_1}$}
\end{equation}
\begin{equation}
	\text{And, for every }  x\in \rn  \,, \text{ either    }  \#\Ga (x)=\infty \text{   or } \Ga (x)=\{x\}. \label{(A_2)} \tag{$\mathbf{A_2}$}
\end{equation}
Here we note that the condition \eqref{(A_1)} is a stricter condition than the one taken in \cite{Clapp_Rios}. In fact, this is necessary as we are working on the Poincar\'e ball.\\

Let $X$ be one of the three Sobolev spaces $H_0^1\(\Bee\), D^{1,2}\(\rn\)$, and $H^1\(\bn\)$. Then the collection of $\phi$-equivariant functions in $X$ forms an infinite-dimensional closed subspace, we denote it by
\begin{align*}
	X^\phi := \{ u \in X : u \; \text{is} \; \phi \text{-equivariant}\}.
\end{align*}
The condition \eqref{(A_1)} ensures that the Hilbert spaces $X^\phi$ are infinite dimensional (See \cite{Cingolani_Clapp_Secchi} and \cite{Manna_Manna1}). \\

We conclude this subsection by recalling two useful consequences of the group action, which will be used throughout the article. The first provides a dichotomy for sequences in $\bn$ in terms of the behavior of their $\Ga$-orbits, while the second shows that the equivariant subspace $H^1\(\bn\)^\phi$ is preserved under the hyperbolic translations by the points belonging to $\mathrm{Fix}_{\bn}(\Ga)$. The proofs can be found in \cite{Manna_Manna1}.\\

\begin{lemma}\label{Dichotomy result} 
	Assume that condition $(\ref{(A_2)})$ holds. For any sequence $\{x_n\} \subset \bn$, up to a subsequence, the following dichotomy holds:\\
	\begin{enumerate}
		\item For every $k \in \N$, there exists a subset $\{g_1, \cdots, g_k\}$ of $\Ga$ such that
		\begin{align*}
			d_{\bn} \(g_i x_n, g_j x_n\) \rightarrow &\infty, \; \text{as} \; n \rightarrow \infty, \; \text{for} \; i \neq j.
		\end{align*}
		\item There exists a sequence $\{b_n\} \subset \bn$ and a constant $R>0$ such that $\{b_n\} \subset \mathrm{Fix}_{\bn}(\Ga)$ and
		\begin{align*}
			d_{\bn} \(b_n, \Ga (x_n)\) & < R, \quad \forall \, n \in \N.
		\end{align*}
	\end{enumerate}
\end{lemma}

\begin{lemma}\label{Equivariant subspaces are closed under translations through the fixed points}
	Consider an element $b$ from $ \bn$ and let $\tau_b$ represent the hyperbolic translation by $b$. If $b \in \mathrm{Fix}_{\bn}(\Ga)$, then for each $u$ in $H^1(\bn)^{\phi}$, it follows that $u \circ \tau_b$ is also in $H^1(\bn)^{\phi}$.
\end{lemma}

\subsection{Representation of a compact group and Bochner integral}

Let $\Ga \subset O(N)$ be a compact subgroup. Let us denote $I$ be the identity matrix in $\Ga$ and $\mathcal{I}$ be the identity map on $H^1\(\bn\)$. Now, we recall:

\medskip

\begin{definition}[Representation of a topological group]
	Let $X$ be a topological vector space, and $G$ a topological group. A representation of $G$ over $X$ is a map $T : G \to BL(X)$ (set of bounded linear maps) such that:
	
	\medskip
	
	\begin{enumerate}
		\item $T$ is a group homomorphism:
		\begin{align*}
			T(I) = \mathcal{I}, \quad T(g_1 g_2) = T(g_1) T(g_2), \quad \forall g_1, g_2 \in G;
		\end{align*}
		\item $T$ is strongly continuous:
		\begin{align*}
			\lim\limits_{g \to I} T(g) x = x, \quad \forall x \in X.
		\end{align*}
	\end{enumerate}
\end{definition}

For our study, let us define $T : \Ga \to BL\(H^1\(\bn\)\)$ such that
\begin{align}
	T(g)(u)(\cdot) := T_g(u)(\cdot) = \phi(g) u\(g^{-1} \cdot \), \quad \forall g \in \Ga, \; \text{and} \; u \in H^1\(\bn\). \label{Representation definition}
\end{align}
Using \Cref{change of variable}, we have that $\norm{u}_\la = \norm{T_gu}_\la, \; \forall g \in \Ga$ and $u \in H^1\(\bn\)$, i.e., the function is well defined. Now, we have the following lemma (the proof is in \Cref{Appendix1}):

\medskip

\begin{lemma}\label{Representation}
	$T$ is a representation of $\Ga$ over $H^1\(\bn\)$.
\end{lemma}

Now, we note that the $\phi$-equivariant subspace $H^1\(\bn\)^\phi$ can be realized as the fixed point set of the representation $T$, i.e.,
\begin{align*}
	H^1\(\bn\)^\phi = \{ u \in H^1\(\bn\) : T_g u = u, \; \forall g \in \Ga\}.
\end{align*}

Since $\Ga$ is a compact topological group, there exists a unique Borel probability measure $\mu$ such that for every complex-valued continuous function $f$ on $\Ga$, we have the following:
\begin{align*}
	& \int_{\Ga} f \(kg\) \mathrm{d}\mu(g) = \int_{\Ga} f \(gk\) \mathrm{d}\mu(g) = \int_{\Ga} f \(g\) \mathrm{d}\mu(g), \; \forall k \in \Ga;\\
	\text{and} \; & \int_{\Ga} f(g) \mathrm{d}\mu(g) = \int_{\Ga} f(g^{-1}) \mathrm{d}\mu(g).
\end{align*}
This measure is called the \textbf{Haar measure} on $\Ga$.\\

For the representation $T$ of $\Ga$ over $H^1\(\bn\)$, it can be shown that the map $(g,u) \mapsto T_g u$ from $\Ga \times H^1\(\bn\)$ into $H^1\(\bn\)$ is continuous (See Lemma 2.5.3 in \cite{Vanderbauwhede}). This implies that for every $u \in H^1\(\bn\)$, the integral $\int_{\Ga} T_g u \mathrm{~d}\mu(g)$ is well defined. Now, we introduce a bounded linear map $P_\phi : H^1\(\bn\) \to H^1\(\bn\)$, defined by
\begin{align}
	P_\phi (u) := \int_{\Ga} T_g u \mathrm{~d}\mu (g). \label{Projection map}
\end{align}
This is a projection of $H^1\(\bn\)$ onto $H^1\(\bn\)^\phi$. The existence follows from Theorem 2.5.13 in \cite{Vanderbauwhede}.\\

Now, we present the notion of Bochner integrability and a few of its properties, which will be used in the forthcoming section (details can be found in the book by Yosida \cite{Yosida}). We fix $\(\Ga, \mathcal{B}, \mu\)$ be the measure space on $\Ga$, where $\mathcal{B}$ is Borel $\sigma$-algebra and $\mu$ is the Haar measure.

\medskip

\begin{definition}
	A function $f : \Ga \to H^1\(\bn\)$ is said to be Bochner $\mu$-integrable if there exists a sequence of finitely-valued functions $f_n : \Ga \to H^1\(\bn\)$ such that the followings hold:
	
	\medskip
	
	\begin{enumerate}
		\item $f_n$ converges to $f$ $\mu$-a.e. on $\Ga$;\\
		\item $\lim\limits_{n \to \infty} \int_{\Ga} \norm{f_n(g) - f(g)}_{H^1\(\bn\)} \mathrm{d} \mu(g) = 0$.
	\end{enumerate}
\end{definition}

\medskip

If $f$ satisfies only the first condition, then it is called a strongly $\mathcal{B}$-measurable function. Now, we recall the following results:

\medskip

\begin{enumerate}
	\item[$\mathbf{R_1}$:] A strongly $\mathcal{B}$-measurable function $f : \Ga \to H^1\(\bn\)$ is Bochner $\mu$-integrable iff $\norm{f}_{H^1\(\bn\)}$ is $\mu$-integrable, and in that case we have
	\begin{align}
		\norm{\int_{\Ga} f(g) \mathrm{~d} \mu(g)}_{H^1\(\bn\)} \leq \int_{\Ga} \norm{f(g)}_{H^1\(\bn\)} \mathrm{d} \mu(g). \label{triangle inequality for Bochner integral}
	\end{align}
	\item[$\mathbf{R_2}$:] Let $T \in \(H^1\(\bn\)\)^*$ and let $f : \Ga \to H^1\(\bn\)$ be Bochner $\mu$-integrable, then
	\begin{align}
		T \int_{\Ga} f(g) \mathrm{~d} \mu(g) = \int_{\Ga} T f(g)  \mathrm{~d} \mu(g). \label{Dual elements commute with Bochner integration}
	\end{align}
\end{enumerate}

\medskip

Now, we present a lemma which shows that the integral in \eqref{Projection map} is a Bochner integral.\\

\begin{lemma}\label{T is Bochner integrable}
	Let $T$ be as in \eqref{Representation definition}, then for a fixed $u \in H^1\(\bn\)$, the map $g \mapsto T_g u$ is Bochner $\mu$-integrable.
\end{lemma}

\begin{proof}
	Since $T$ is a representation of $\Ga$ over $H^1\(\bn\)$, the map $(g,u) \mapsto T_g u$ from $\Ga \times H^1\(\bn\)$ into $H^1\(\bn\)$ is continuous. Let us now fix $n \in \N$, then for each $g \in \Ga$, from the continuity we have an open neighborhood $U_{n,g}$ such that 
	\begin{align*}
		\norm{T_gu - T_ku}_\la < \f{1}{n}, \quad \forall k \in U_{n,g}.
	\end{align*}
	Now $\{U_{n,g} : g \in \Ga\}$ is an open cover of $\Ga$, from which we can choose finitely many such that $\Ga = U_{n, g_n^1} \cup \cdots \cup U_{n, g_n^i}$. Let us define
	\begin{align*}
		B_{n,1} = U_{n. g_n^1} \; \text{and} \; B_{n,j} = U_{n,g_n^j} \setminus \cup_{l=1}^{j-1} U_{n,g_n^l}, \; \text{for} \; 2 \leq j \leq i.
	\end{align*}
	Note that all $B_{n,j}$ are measurable.  Now let us define a finitely valued function on $\Ga$ to be $S_n u = \sum_{j=1}^{i} T_{g_n^j}  u \chi_{B_{n,j}}.$ Then we observe that
	\begin{align*}
		\sup_{g \in \Ga} \norm{T_g u - S_nu(g)}_\la \leq \f{1}{n}.
	\end{align*}
	This implies the map $g \mapsto T_g u$ is strongly $\mu$-measurable. Since for each $g \in \Ga$, $T_g$ is an isometry on $H^1\(\bn\)$, we have
	\begin{align*}
		\int_{\Ga} \norm{T_gu}_\la \mathrm{~d}\mu (g) = \int_{\Ga} \norm{u}_\la \mathrm{~d}\mu (g) = \norm{u}_\la < \infty.
	\end{align*}
	Hence, the map $g \mapsto T_g u$ is Bochner $\mu$-integrable.
\end{proof}

\section{Limit problem and energy estimates}\label{Limit problem}

Throughout this section, we take $N \geq 4$ and $\la \in \(\f{N(N-2)}{4}, \f{(N-1)^2}{4}\)$. Before going into the proof of the existence theorem, we discuss some key results related to the limiting critical problem and establish an energy estimate that will be used in the subsequent analysis for the compactness. The loss of compactness in our problem is governed by the critical equation in the Euclidean space
\begin{align}
	- \Delta w = |w|^{2^*-2} w, \quad w \in D^{1,2}(\rn). \label{equation in R^N111} \tag{$Eq^\infty$}
\end{align}
In \cite{Clapp_Rios}, Clapp and Rios studied this limit problem as a particular case of the critical quasi-linear equation:
\begin{align*}
	- \De_q u = |u|^{q^*-2}u, \;\; u \in D^{1,q}(\rn),
\end{align*} 
where $1<q<N$ and $q^* = \f{Nq}{N-q}$. We will only use their result for the semi-linear case, i.e., $q=2$. Their main theorem is the following:\\

\begin{theorem}[Theorem 3.1, \cite{Clapp_Rios}]
	Let $\Ga$ be a closed subgroup of $O(N)$ and $\phi : \Ga \to \mathbb{Z}_2$ be a continuous onto homomorphism which satisfies \eqref{(A_1)} and \eqref{(A_2)}. Then \eqref{equation in R^N} possesses a nontrivial solution in $D^{1,2}\(\rn\)^\phi$.
\end{theorem}

\medskip

The corresponding energy functional restricted to $D^{1,2}\(\rn\)^\phi$ is
\begin{align*}
	J_{\infty}(w) = \f{1}{2} \int_{\rn} \abs{\nabla w}^2 \dx - \f{1}{2^*} \int_{\rn} |w|^{2^*} \dx, \quad w \in D^{1,2}\(\rn\)^\phi.
\end{align*}
Then the Nehari manifold is defined as
\begin{align*}
	\mathcal{N}_\infty^{\phi}(\rn) := \{w \in D^{1,2}\(\rn\)^\phi \setminus \{0\} \; : \; \int_{\rn} \abs{\nabla w}^2 \dx = \int_{\rn} |w|^{2^*} \dx \}.
\end{align*}
They showed that there exists a minimizer $W$ for $c_{\infty}^\phi := \inf_{w \in \mathcal{N}_\infty^{\phi}(\rn)} J_\infty(w).$ From the principle of symmetric criticality, $W \in D^{1,2}\(\rn\)^\phi$ solves \eqref{equation in R^N}. For a fixed pair $\(\Ga, \phi\)$, we denote by $S_{\infty}^\phi$ the equivariant Sobolev level associated with \eqref{equation in R^N} such that
\begin{align*}
	S_{\infty}^\phi := \inf_{w \in D^{1,2}\(\rn\)^\phi \setminus \{0\}} \f{\int_{\rn} \abs{\nabla w}^2 \dx}{ \(\int_{\rn} |w|^{2^*} \dx\)^{\f{2}{2^*}}}.
\end{align*}
We observe that $S_{\infty}^\phi = \(N c_{\infty}^\phi \)^{\f{2}{N}}$. Therefore, $W$ is also a minimizer for $S_{\infty}^\phi$. Now, for the same fixed pair $\(\Ga, \phi\)$ and $\la \in \(\f{N(N-2)}{4}, \f{(N-1)^2}{4}\)$ we define the equivariant Poincar\'e-Sobolev level to be
\begin{align*}
	S^\phi := \inf_{u \in H^1\(\bn\)^\phi \setminus \{0\}} \f{\norm{u}_\la^2}{\(\int_{\bn} |u|^{2^*} \dv \)^{\f{2}{2^*}}}.
\end{align*}
From the conformal change of metric, we have
\begin{align*}
	S^\phi = \inf_{v \in H_0^1(\Bee)^\phi \setminus \{0\}} \f{Q(v)}{\(\int_{\Bee} |v|^{2^*} \dx \)^{\f{2}{2^*}}} = \inf_{v \in H_0^1(\Bee)^\phi \setminus \{0\}} \f{\int_{\Bee} \abs{\nabla v}^2 \dx - \tilde{\la} \int_{\Bee} h^2 v^2 \dx}{\(\int_{\Bee} |v|^{2^*} \dx \)^{\f{2}{2^*}}},
\end{align*}
where $\tilde{\la} = \la - \f{N(N-2)}{4}$. To show the existence of a nonradial sign-changing solution we will show that $S^\phi$ can be achieved by a function either from $H_0^1(\Bee)^\phi$ or $H^1\(\bn\)^\phi$. The main difficulty is that the critical Sobolev embedding is not compact, so a minimizing sequence for $S^\phi$ may develop a concentration. The purpose of this section is to prove the strict comparison $S^\phi < S_\infty^\phi$. This strict inequality will be essential to exclude concentration phenomena and recover compactness of the minimizing sequences. To show this inequality, we exploit the equivariance of $W$ to obtain an improved decay estimate at infinity and then use suitably truncated rescalings of $W$ as test functions for $S^\phi$.

As $W$ solves \eqref{equation in R^N}, from the elliptic regularity we have that $W \in C^{1, \al}\(\rn\)$ for some $0<\al < 1$. The standard decay of a finite energy solution gives (see \cite{Vetois})
\begin{align*}
	|W(x)| \leq C_0 \(1+|x|^{N-2}\)^{-1}, \;\; \text{and} \; \abs{\nabla W(x)} \leq C_0 \(1+|x|^{N-1}\)^{-1},
\end{align*}
for every $x \in \rn$, i.e.,
\begin{align*}
	W(x) = O\(|x|^{2-N}\) \;\; \text{and} \;\; \nabla W(x) = O \(|x|^{1-N}\), \;\; \text{as} \; |x| \to \infty.
\end{align*} 
However, the equivariance of $W$ and the fact that \eqref{equation in R^N} is invariant under Kelvin transformation allow us to improve these estimates by one order.\\

\begin{lemma}\label{improved decay estimate}
	Let $W \in D^{1,2}\(\rn\)^\phi$ be a solution to \eqref{equation in R^N}, then
	\begin{align*}
		W(x) = O\(|x|^{1-N}\) \;\; \text{and} \;\; \nabla W(x) = O \(|x|^{-N}\), \;\; \text{as} \; |x| \to \infty.
	\end{align*}
\end{lemma}

\begin{proof}
	The Kelvin transformation for $W$ is
	\begin{align*}
		\hat{W}(y) = |y|^{2-N} W \(\f{y}{|y|^2}\), \quad y \neq 0.
	\end{align*}
	Since $W \in D^{1,2}(\rn)$, the conformal invariance guarantees that $\hat{W} \in D^{1,2}(\rn)$ (see \cite{Egnell}) and it solves 
	\begin{align}
		- \De \hat{W} = \abs{\hat{W}}^{2^*-2} \hat{W}, \quad \text{in} \; \rn \setminus \{0\}. \label{Kelvin transformed equation}
	\end{align}
	\textbf{Claim I:} $\hat{W} \in C^1_{loc} \(\Bee(0,R)\)$, for any $R>0$.\\
	\newline
	To prove this, first we prove that $\hat{W}$ is a weak solution to the above equation in $\Bee(0,R)$, then elliptic regularity will complete the claim. Let $\varphi \in C_c^\infty\(\Bee(0,R)\)$ and a smooth cutoff function $\eta_\eps$ such that
	\begin{align*}
		\eta_\eps = \begin{cases}
			0, \quad & \text{in} \; \Bee(0,\eps)\\
			1,  & \text{outside} \; \Bee(0,2\eps)
		\end{cases} \quad \text{and} \; \abs{\nabla \eta_\eps} \leq \f{C}{\eps}, \; \text{for} \; C >0.
	\end{align*}
	Now $\eta_\eps \varphi$ is an admissible test function for the equation \eqref{Kelvin transformed equation}, i.e.,
	\begin{align*}
		\int_{\Bee(0,R)} \eta_\eps \(\nabla \hat{W} \cdot \nabla \varphi\) \dx + \int_{\Bee(0,2\eps) \setminus \Bee(0,\eps)} \varphi \( \nabla \hat{W} \cdot \nabla \eta_\eps\) \dx =  \int_{\Bee(0,R)} \abs{\hat{W}}^{2^*-2} \hat{W} \eta_\eps \varphi \dx.
	\end{align*}
	Now
	\begin{align*}
		\abs{\int_{\Bee(0,2\eps) \setminus \Bee(0,\eps)} \varphi \( \nabla \hat{W} \cdot \nabla \eta_\eps\) \dx} & \leq \norm{\varphi}_{L^\infty} \norm{\nabla \hat{W}}_{L^2\(\Bee(0,2\eps) \setminus \Bee(0,\eps)\)} \(\int_{\Bee(0,2\eps) \setminus \Bee(0,\eps)} \abs{\nabla \eta_\eps}^2 \dx\)^{\f{1}{2}} \leq C \(\eps^{N-2}\)^{\f{1}{2}}
	\end{align*}
	Therefore, taking $\eps \to 0$, from the dominated convergence theorem, we have that
	\begin{align*}
		\int_{\Bee(0,R)} \nabla \hat{W} \cdot \nabla \varphi \dx = \int_{\Bee(0,R)} \abs{\hat{W}}^{2^*-2} \hat{W} \varphi \dx.
	\end{align*}
	Hence, \textbf{Claim I} follows.\\
	\newline
	\textbf{Claim II:} $\hat{W}(0) = 0$.\\
	\newline
	For any $y \neq 0$ and $g \in \Ga$, we have
	\begin{align*}
		\hat{W}(gy) = |y|^{2-N} W \(g \(\f{y}{|y|^2}\)\) = \phi(g) |y|^{2-N} W \(\f{y}{|y|^2}\) = \phi(g) \hat{W}(y).
	\end{align*}
	Now taking $y \to 0$, we have $\hat{W}(0) = \phi(g) \hat{W}(0)$. Since $\phi$ is onto, we take $g \in \Ga$ such that $\phi(g) = -1$, that implies $\hat{W}(0) = 0$. This completes the proof of \textbf{Claim II}.\\
	
	Now, from the mean value theorem, we have
	\begin{align*}
		& \abs{\hat{W}(y)} \leq C|y|, \quad \text{as} \; |y| \to 0.\\
		\Rightarrow & \abs{W(x)} \leq C |x|^{2-N} \cdot \f{1}{|x|} = O \(|x|^{1-N}\), \quad \text{as} \; |x| \to \infty.
	\end{align*}
	Furthermore, as $x\to \infty$ we obtain
	\begin{align*}
		\abs{\nabla W(x)} & \leq (N-2) |x|^{1-N} \abs{\hat{W}\(\f{x}{|x|^2}\)} + |x|^{2-N} \abs{\nabla \hat{W} \(\f{x}{|x|^2}\)} \cdot \f{1}{|x|^2} \leq C |x|^{-N}.
	\end{align*}
\end{proof}

As an immediate consequence of the previous lemma, we have $W \in L^2\(\rn\)$, for every $N \geq 4$. Now, we use $W$ to construct a family of test functions inside $H_0^1\(\Bee\)^\phi$ for the energy level $S^\phi$. Let $\eta \in C_c^\infty\(\Bee\)$ be a radial cutoff function such that
\begin{align*}
	0 \leq \eta \leq 1, \quad \eta = \begin{cases}
		1 \quad & \text{in} \; B(0,\rho)\\
		0 & \text{outside} \; B(0,2\rho)
	\end{cases}, 
\end{align*}
for some $ 0 < \rho < \f{1}{3}.$ Now for $\eps > 0$, we define
\begin{align*}
	W_\eps(x) = \eps^{\f{2-N}{2}} \eta(x) W\(\f{x}{\eps}\), \quad x \in \Bee.
\end{align*}
Since $\eta$ is radial and $\Ga \subset O(N)$, it is $\Ga$-invariant. Therefore, $\phi$-equivariancy of $W$ is preserved under this localization, i.e.,
\begin{align*}
	W_\eps(gx) = \phi(g) W_\eps(x), \quad \forall x \in \Bee, g \in \Ga.
\end{align*}
Moreover, $\mathrm{Supp} W_\eps \Subset \Bee$ and \Cref{improved decay estimate} yields the required integrability to conclude that $W_\eps \in H_0^1\(\Bee\)^\phi$. In the following lemma, we estimate all three terms in the quotient defining $S^\phi$.\\

\begin{lemma}
	As $\eps \to 0$, 
	\begin{align}
		(i) \; & \int_{\Bee} \abs{\nabla W_\eps}^2 \dx \leq \int_{\rn} \abs{\nabla W}^2 \dx + O(\eps^N). \label{estimate for gradient term}\\
		(ii) \; & \int_{\Bee} h^2 W_\eps^2 \dx = 4 \eps^2 \int_{\rn} W^2 \dx + o(\eps^2). \label{estimate for square term}\\
		(iii) \; & \int_{\Bee} \abs{W_\eps}^{2^*} \dx = \int_{\rn} \abs{W}^{2^*} \dx + o(\eps^2). \label{estimate for critical term}
	\end{align}
\end{lemma}

\begin{proof}
	$(i)$ 
	\begin{align*}
		\int_{\Bee} \abs{\nabla W_\eps}^2 \dx = & \eps^{-N} \int_{\Bee} \eta(x)^2 \abs{\nabla W \(\f{x}{\eps}\)}^2 \dx + \\
		& \eps^{2-N} \int_{\Bee} \abs{\nabla \eta (x)}^2 W \(\f{x}{\eps}\)^2 \dx + 2 \eps^{1-N} \int_{\Bee} \eta(x) W \(\f{x}{\eps}\) \(\nabla \eta(x) \cdot \nabla W \(\f{x}{\eps}\)\) \dx\\
		= & \int_{\rn} \eta(\eps y)^2 \abs{\nabla W(y)}^2 \dy + \eps^2 \int_{\rn} \abs{\nabla \eta (\eps y)}^2 W(y)^2 \dy + \\
		& \hspace*{2cm}  2\eps \int_{\rn} \eta(\eps y) W(y) \( \nabla \eta (\eps y) \cdot \nabla W(y)\) \dy\\
		=: & \; I_1 + I_2 + I_3.
	\end{align*}
	Clearly, we have $I_1 \leq \int_{\rn} \abs{\nabla W}^2 \dx.$ Now, using the \Cref{improved decay estimate} we have
	\begin{align*}
		I_2 \leq \eps^2 \int_{\f{\rho}{\eps} < |y| < \f{2\rho}{\eps}} W(y)^2 \dy \leq C \eps^2 \int_{\f{\rho}{\eps}}^{\f{2\rho}{\eps}} r^{2(1-N)} \cdot r^{N-1} \mathrm{~d}r = C \eps^2 \int_{\f{\rho}{\eps}}^{\f{2\rho}{\eps}} r^{1-N} \mathrm{~d}r \leq C \eps^N = O(\eps^N).
	\end{align*}
	Again, using the \Cref{improved decay estimate} we obtain
	\begin{align*}
		I_3 & \leq 2 \eps \int_{\{\f{\rho}{\eps} < |y| < \f{2\rho}{\eps}\}} \eta(\eps y) \abs{W(y)} \abs{\nabla \eta \(\eps y\)} \abs{\nabla W(y)} \dy\\
		& \leq C \eps \int_{\f{\rho}{\eps}}^{\f{2\rho}{\eps}} r^{1-N} \cdot r^{-N} \cdot r^{N-1} \mathrm{~d}r = C \eps \int_{\f{\rho}{\eps}}^{\f{2\rho}{\eps}} r^{-N} \mathrm{~d}r \leq C \eps^N = O(\eps^N).
	\end{align*}
	\vspace*{3mm}
	
	$(ii)$ First we observe that
	\begin{align*}
		\int_{\Bee} h(x)^2 W_\eps(x)^2 \dx = \eps^{2-N} \int_{\Bee} h(x)^2 \eta(x)^2 W \(\f{x}{\eps}\)^2 \dx = \eps^2 \int_{\rn} h(\eps y)^2 \eta(\eps y)^2 W(y)^2 \dy.
	\end{align*}
	Now, for each fixed $y \in \rn$, $\eps y \to 0$ as $\eps \to 0$. Then $h(\eps y) \to 2$ and $\eta(\eps y) \to 1$ pointwise in $\rn$ as $\eps \to 0$. This implies that as $\eps \to 0$, we get $h(\eps y)^2 \eta(\eps y)^2 W(y)^2 \to 4 W(y)^2 \; \text{pointwise.}$ As $\mathrm{Supp} \eta \subset \Bee(0,2\rho)$, where $\rho < \f{1}{3}$, there exists $M >0$ such that $h(x) \eta(x) \leq M$ in $B$, i.e., we will have $h(\eps y) \eta (\eps y) \leq M$ in $\rn$. Therefore $h(\eps y)^2 \eta(\eps y)^2 W(y)^2 \leq M^2 W(y)^2$ and from the \Cref{improved decay estimate}, we have $W \in L^2\(\rn\)$. Hence, from the dominated convergence theorem, as $\eps \to 0$, we obtain
	\begin{align*}
		& \int_{\rn} h(\eps y)^2 \eta(\eps y)^2 W(y)^2 \dy = 4 \int_{\rn} W(y)^2 \dy + o(1)\\
		\Rightarrow & \; \int_{\Bee} h(x)^2 W_\eps(x)^2 \dx = \eps^2 \int_{\rn} h(\eps y)^2 \eta(\eps y)^2 W(y)^2 \dy = 4 \eps^2 \int_{\rn} W(y)^2 \dy + o(\eps^2).
	\end{align*}
	
	\vspace*{3mm}
	
	$(iii)$
	\begin{align*}
		\int_{\Bee} \abs{W_\eps}^{2^*} \dx = \eps^{-N} \int_{\Bee} \eta(x)^{2^*} \abs{W\(\f{x}{\eps}\)}^{2^*} \dx = \int_{\rn} \eta(\eps y)^{2^*} \abs{W(y)}^{2^*} \dy.
	\end{align*}
	Since $W \in L^{2^*}\(\rn\)$, using the dominated convergence theorem, as $\eps \to 0$ we have
	\begin{align*}
		\int_{\Bee} \abs{W_\eps}^{2^*} \dx = \int_{\rn} \abs{W}^{2^*} \dx + o(1).
	\end{align*}
	Now, we estimate the difference between two quantities:
	\begin{align*}
		\abs{\int_{\Bee} \abs{W_\eps}^{2^*} \dx - \int_{\rn} \abs{W}^{2^*} \dx} & \leq \int_{\{|y| > \f{\rho}{\eps}\}} \abs{W(y)}^{2^*} \dy\\
		& \leq C \int_{\f{\rho}{\eps}}^{\infty} r^{(1-N)(2^*-1)} \mathrm{~d}r = O \(\eps^{(N-1)(2^*-1) - 1}\) = o(\eps^2).
	\end{align*}
	Therefore
	\begin{align*}
		\int_{\Bee} \abs{W_\eps}^{2^*} \dx = \int_{\rn} \abs{W}^{2^*} \dx + o(\eps^2).
	\end{align*}
\end{proof}

We now use the above estimates to compare the energy level $S^\phi$ for conformal equation \eqref{conformal equation} and the limiting level $S_\infty^\phi$ for \eqref{equation in R^N}.\\

\begin{lemma}\label{local compactness level}
	$S^\phi < S^\phi_{\infty}$
\end{lemma}

\begin{proof}
	For each $\eps>0$, $W_\eps \in H_0^1(\Bee)^\phi \setminus \{0\}$. Using the estimates \eqref{estimate for gradient term}, \eqref{estimate for square term}, and \eqref{estimate for critical term}, we have
	\begin{align*}
		S^\phi  \leq \f{\int_{\Bee} \abs{\nabla W_\eps}^2 \dx - \tilde{\la} \int_{\Bee} h^2 W_\eps^2 \dx}{\(\int_{\Bee} \abs{W_\eps}^{2^*} \dx \)^{\f{2}{2^*}}} & \leq \f{ \int_{\rn} \abs{\nabla W}^2 \dx - 4 \eps^2 \tilde{\la} \int_{\rn} W^2 \dx + o(\eps^2) }{ \( \int_{\rn} \abs{W}^{2^*} \dx + o(\eps^2) \)^{\f{2}{2^*}}  }\\
		& = \f{ \int_{\rn} \abs{\nabla W}^2 \dx - 4 \eps^2 \tilde{\la} \int_{\rn} W^2 \dx + o(\eps^2) }{ \( \int_{\rn} \abs{W}^{2^*} \dx \)^{\f{2}{2^*}} + o(\eps^2) }\\
		& = \f{\int_{\rn} \abs{\nabla W}^2 \dx }{ \(\int_{\rn} \abs{W}^{2^*} \dx\)^{\f{2}{2^*}}} - \f{4 \tilde{\la} \int_{\rn} W^2 \dx}{ \(\int_{\rn} \abs{W}^{2^*} \dx\)^{\f{2}{2^*}}} \eps^2 + o(\eps^2)
	\end{align*}
	Since $\tilde{\la} > 0$, for sufficiently small $\eps>0$, we have
	\begin{align*}
		S^\phi < \f{\int_{\rn} \abs{\nabla W}^2 \dx }{ \(\int_{\rn} \abs{W}^{2^*} \dx\)^{\f{2}{2^*}}} = S^\phi_\infty.
	\end{align*}
\end{proof}

\section{Existence result}\label{section for existence}

In this section, again we take $N \geq 4$ and $\la \in \(\f{N(N-2)}{4}, \f{(N-1)^2}{4}\)$. Firstly, for a fixed pair $\(\Ga, \phi\)$ satisfying \eqref{(A_1)} and \eqref{(A_2)}, we show that $S^\phi$ is achieved in $H_0^1\(\Bee\)^\phi$ or $H^1\(\bn\)^\phi$, using the strict energy estimate $S^\phi < S_\infty^\phi$. Then we prove our main result: \Cref{Existence theorem}, by combining the attainment result and by giving examples of different pairs $\(\Ga_j, \phi_j\)$, $j \in \{1, \cdots, \lfloor \f{N}{4} \rfloor \} $, so that the corresponding equivariant subspaces will have trivial intersection. For the groups $\Ga$ of our consideration, there are two possible scenarios: (i) $\mathrm{Fix}_{\rn} \Ga = \{0\}$, (ii) $\mathrm{Fix}_{\rn}\Ga$ becomes a nontrivial vector subspace of $\rn$. In both cases, we will prove the fundamental objective: ``\textbf{whether $S^\phi$ can be achieved?}", which are studied separately in the following two subsections:

\subsection{Case I: $\mathrm{Fix}_{\rn} \Ga = \{0\}$}

If $\mathrm{Fix}_{\rn} \Ga = \{0\}$, we work with the conformal equation \eqref{conformal equation} on $\Bee$ and use concentration compactness principle by Lions \cite{Lions limit case 1} together with the condition \eqref{(A_2)} to recover the compactness of minimizing sequence of $S^\phi$. The concentration compactness principle for the bounded sequences in $H_0^1\(\Bee\)$ is as follows:\\

\begin{lemma}\label{Lions}
	Let $\{u_n\}$ be a bounded sequence in $H_0^1(\Omega)$ for some bounded domain $\Omega$ of $\rn$. We may assume that $\{u_n\}$ converges weakly in $H_0^1\(\Omega\)$ to some $u$ and that $|u_n|^{2^*}$ converges weakly in the sense of measures to some $\nu$. Then there exists $\{x_i\}_{i \geq 1}$ in $\overline{\Omega}$ and $\{\nu_i\}_{i \geq 1}$ in $[0,\infty)$ such that
	\begin{align*}
		\nu = |u|^{2^*} + \sum_{i=1}^{\infty} \nu_i \delta_{x_i}, \quad \sum_{i=1}^{\infty} \nu_i^{\f{N-2}{N}} < \infty.
	\end{align*}
\end{lemma} 

\begin{theorem}\label{Attainment result when fixed point is trivial}
	Let $\(\Ga, \phi\)$ satisfy \eqref{(A_1)}, \eqref{(A_2)} and $\mathrm{Fix}_{\rn} \Ga = \{0\}$. Then $S^\phi$ is attained in $H_0^1\(\Bee\)^\phi$.
\end{theorem}

\begin{proof}
	Let $\{v_n\} \subset H_0^1\(\Bee\)^\phi$ be a minimizing sequence for $S^\phi$ such that 
	\begin{align}
		\int_{\Bee} \abs{v_n}^{2^*} \dx = 1, \quad \text{and} \quad Q(v_n) \to S^\phi. \label{v_n is minimizing}
	\end{align}
	After passing to a subsequence $v_n \rightharpoonup v$ in $H_0^1\(\Bee\)^\phi$, i.e., in $H_0^1\(\Bee\)$. Set $z_n := v_n - v$, then $z_n \rightharpoonup 0$ in $H_0^1\(\Bee\)$ and $z_n \in H_0^1\(\Bee\)^\phi$. From the Br\'ezis-Lieb lemma, we have
	\begin{align*}
		1 = \int_{\Bee} \abs{v_n}^{2^*} \dx = \int_{\Bee} |v|^{2^*} \dx + \int_{\Bee} |z_n|^{2^*} \dx + o(1).
	\end{align*} 
	Furthermore, since $Q(\cdot)^{\f{1}{2}}$ is a norm in $H_0^1(\Bee)$ we have
	\begin{align}
		Q(v_n) = Q(v) + Q(z_n) + o(1). \label{quadratic term decomposition}
	\end{align}
	Let us set 
	\begin{align*}
		\al := \int_{\Bee} |z_n|^{2^*} \dx + o(1) \quad \text{and} \quad \beta := \int_{\Bee} |v|^{2^*} \dx.
	\end{align*}
	Note that $\al, \beta \in [0,1]$ and $\al + \beta = 1$. Now, from the definition of $S^\phi$, we have 
	\begin{align}
		Q(v) \geq S^\phi \( \int_{\Bee} |v|^{2^*} \dx \)^{\f{2}{2^*}} = S^\phi \beta^{\f{2}{2^*}}. \label{inequality for Q(v)}
	\end{align}
	Now, by Lions' concentration compactness lemma (\Cref{Lions}) we argue that $|z_n|^{2^*}$ weakly converges in the sense of measure to some $\nu$ and there exists $\{x_i\}_{i \geq 1}$ in $\overline{\Bee}$ and $\{\nu_i\}_{i \geq 1}$ in $[0,\infty)$ such that $\nu = \sum_{i=1}^{\infty} \nu_i \delta_{x_i}.$\\
	
	We take the zero extension of each $z_n$, and we have $z_n \in D^{1,2}\(\rn\)^\phi$. Now, we show that $\nu$ is $\Ga$-invariant, i.e., $\nu(gE) = \nu(E)$, for every Borel set $E \subset \overline{\Bee}$ and $g \in \Ga$. First, we note that since $z_n$ is $\phi$-equivariant, $|z_n|^{2^*}$ is $\Ga$-invariant. Let us set $\nu^n := |z_n|^{2^*} \dx$, then for any Borel set $E \subset \overline{\Bee}$, we get
	\begin{align*}
		\nu^n \(gE\) = \int_{gE} |z_n(x)|^{2^*} \dx = \int_{E} |z_n(gy)|^{2^*} \dy = \int_{E} |z_n(y)|^{2^*} \dy = \nu^n (E). 
	\end{align*}
	Therefore, $\nu^n$ is a $\Ga$-invariant measure. We have that $\nu^n \rightharpoonup \nu$, i.e., as $n \to \infty$, 
	\begin{align*}
		\int_{\Bee} \psi \mathrm{~d}\nu^n \to \int_{\Bee} \psi \mathrm{~d}\nu, \;\; \forall \psi \in C\(\overline{\Bee}\).
	\end{align*}
	Since $\nu^n$ is $\Ga$-invariant, for each $\psi \in C\(\overline{\Bee}\)$ we have
	\begin{align*}
		\int_{\Bee} \psi (gx) \mathrm{~d}\nu^n(x) = \int_{\Bee} \psi(x) \mathrm{~d}\nu^n(x) \;
		\Rightarrow \; \int_{\Bee} \psi (gx) \mathrm{~d}\nu(x) = \int_{\Bee} \psi(x) \mathrm{~d}\nu(x).
	\end{align*}
	Hence $\nu(gE) = \nu(E)$, for every Borel set $E \subset \overline{\Bee}$ and $g \in \Ga$, i.e., $\nu$ is $\Ga$-invariant.\\
	
	The condition \eqref{(A_2)} and $\mathrm{Fix}_{\rn} \Ga = \{0\}$ implies $\# \Ga(x) = \infty$, for every $x \in \rn \setminus \{0\}$, and the Dirac mass can only appear at the origin, which is fixed point for the group $\Ga$. Indeed, if Dirac mass appears at some point $x_1 \(\neq 0\)$, then since $\nu$ is $\Ga$-invariant, we will get Dirac masses for each point in the orbit of $x_1$. Since $\# \Ga(x_1) = \infty$, it will contradict the fact $\norm{z_n}_{L^{2^*} \(\overline{\Bee}\)} \leq 1$. Therefore, no concentration can occur at any $x \in \overline{B} \setminus \{0\}$. Hence, $\nu = \al \delta_0$.\\
	
	Now, let us consider a radial cutoff function $\xi \in C_c^\infty(\Bee)$ such that
	\begin{align*}
		\xi = \begin{cases}
			1 \quad & \text{in} \; \Bee(0,\rho)\\
			0 \quad & \text{outside} \; \Bee(0,2\rho)
		\end{cases}, \quad \text{where} \; \rho < \f{1}{3}.
	\end{align*}
	Then $\xi z_n \in D^{1,2}(\rn)^\phi$ and we have
	\begin{align*}
		\int_{\rn} \abs{\nabla \(\xi z_n\)}^2 \dx \geq S_\infty^\phi \(\int_{\rn} \abs{\xi z_n}^{2^*} \dx \)^{\f{2}{2^*}}.
	\end{align*}
	Using Rellich-Kondrachov theorem, we have $z_n \to 0$ in $L^2\(\Bee(0,2\rho)\)$. Since $h$ is bounded on $\Bee\(0,2\rho\)$, we argue $h \xi z_n \to 0$ in $L^2\(\Bee\)$. This implies
	\begin{align*}
		& Q\(\xi z_n\) = \int_{\rn} \abs{\nabla \(\xi z_n\)}^2 \dx + o(1)\\
		\Rightarrow & \liminf_{n \to \infty} Q\(\xi z_n\) \geq S_\infty^\phi \( \liminf_{n \to \infty} \(\int_{\rn} \abs{\xi z_n}^{2^*} \dx \)^{\f{2}{2^*}}\).
	\end{align*}
	Now, for a small fixed $\rho > 0$, $\lim\limits_{n \to \infty} \int_{\rn} \abs{\xi z_n}^{2^*} \dx = \al$ and this implies
	\begin{align}
		\liminf_{n \to \infty} Q\(\xi z_n\) \geq S_\infty^\phi \al^{\f{2}{2^*}}. \label{liminf of quadratic term with cutoff}
	\end{align}
	Now, we choose another radial cutoff function $\zeta$ such that
	\begin{align*}
		\zeta = \begin{cases}
			0 \quad & \text{in} \; \Bee(0,\rho)\\
			1 \quad & \text{outside} \; \Bee(0,2\rho)
		\end{cases} \quad \text{and} \quad \xi^2 + \zeta^2 = 1.
	\end{align*}
	Then, we have $\abs{\nabla \xi} \leq \f{C}{\rho}$ and $\abs{\nabla \zeta} \leq \f{C}{\rho}$. Now
	\begin{align*}
		\abs{\nabla \(\xi z_n\)}^2 + \abs{\nabla \(\zeta z_n\)}^2 & = \(\xi^2 + \zeta^2\) \abs{\nabla z_n}^2 + \( \abs{\nabla \xi}^2 + \abs{\nabla \zeta}^2\) z_n^2 + z_n \( \nabla z_n \cdot \nabla \(\xi ^2 + \zeta^2\)\)\\
		& = \abs{\nabla z_n}^2 + \( \abs{\nabla \xi}^2 + \abs{\nabla \zeta}^2\) z_n^2
	\end{align*}
	Now using this identity, we write
	\begin{align*}
		Q(z_n) & = \int_{\Bee} \abs{\nabla z_n}^2 \dx - \tilde{\la} \int_{\Bee} h^2 z_n^2 \dx\\
		& = \int_{\Bee} \( \abs{\nabla \(\xi z_n\)}^2 + \abs{\nabla \(\zeta z_n\)}^2 - \( \abs{\nabla \xi}^2 + \abs{\nabla \zeta}^2\) z_n^2 \) \dx - \tilde{\la} \int_{\Bee} \( h^2 \xi ^2 z_n^2 + h^2 \zeta^2 z_n^2\) \dx\\
		& = Q (\xi z_n) + Q \(\zeta z_n\) - \int_{\Bee} \( \abs{\nabla \xi}^2 + \abs{\nabla \zeta}^2\) z_n^2 \dx\\
		& \geq Q (\xi z_n) - \int_{\Bee} \( \abs{\nabla \xi}^2 + \abs{\nabla \zeta}^2\) z_n^2 \dx.
	\end{align*}
	For fixed $\rho > 0$, using the compact embedding $H_0^1\(\Bee\) \hookrightarrow L^2 \(\Bee(0,2\rho) \setminus \Bee(0,\rho)\)$, we have
	\begin{align*}
		\int_{\Bee} \( \abs{\nabla \xi}^2 + \abs{\nabla \zeta}^2\) z_n^2 \dx \to 0.
	\end{align*}
	Therefore, from \eqref{liminf of quadratic term with cutoff}, we obtain
	\begin{align*}
		\liminf_{n \to \infty} Q(z_n) \geq \liminf_{n \to \infty} Q (\xi z_n) \geq S_\infty^\phi \al^{\f{2}{2^*}}.
	\end{align*}
	Now, from \eqref{quadratic term decomposition}, \eqref{v_n is minimizing}, and \eqref{inequality for Q(v)} we argue that
	\begin{align*}
		S^\phi \geq S^\phi \beta^{\f{2}{2^*}} + S_\infty^\phi \al^{\f{2}{2^*}}.
	\end{align*}
	If possible let $\al > 0$, then we use the \Cref{local compactness level} and the above inequality becomes
	\begin{align*}
		S^\phi \geq S^\phi \beta^{\f{2}{2^*}} + S_\infty^\phi \al^{\f{2}{2^*}} > S^\phi \(\beta^{\f{2}{2^*}} +  \al^{\f{2}{2^*}}\) \geq S^\phi (\al + \beta)^{\f{2}{2^*}} = S^\phi.
	\end{align*}
	This is a contradiction. Hence $\al = 0$, i.e.,
	\begin{align*}
		v_n \to v \;\; \text{in} \;\; L^{2^*}\(\Bee\), \quad \text{and} \quad \int_{\Bee} |v|^{2^*} \dx = 1.
	\end{align*}
	Now from the weak lower semicontinuity of $Q$, we have $Q(v) \leq \liminf_{n \to \infty} Q(v_n) = S^\phi$. Hence $Q(v) = S^\phi$ and $\int_{\Bee} |v|^{2^*} \dx = 1$.
\end{proof}

\begin{cor}\label{Corollary for existence in the trivial fixed point case}
		Let $\(\Ga, \phi\)$ satisfy \eqref{(A_1)}, \eqref{(A_2)} and $\mathrm{Fix}_{\rn} \Ga = \{0\}$. Then $S^\phi$ is attained in $H^1\(\bn\)^\phi$.
\end{cor}

\begin{proof}
	Let $v \in H_0^1\(\Bee\)^\phi$ be the minimizer for $S^\phi$. Then from the conformal covariance property, we get that $u := h^{- \f{N-2}{2}} v$ is in $H^1\(\bn\)^\phi$ and it is a minimizer for $S^\phi$.
\end{proof}

\subsection{Case II: $\mathrm{dim} \(\mathrm{Fix}_{\rn} \Ga \) > 0$}

When $\mathrm{dim} \(\mathrm{Fix}_{\rn} \Ga \) > 0$, the preceding argument is no longer sufficient, since concentration may drift along the unbounded fixed point set. In this case, we work directly with the original equation \eqref{(1)} on the entire hyperbolic space $\bn$ and recenter the minimizing sequence via suitable hyperbolic translations. Our analysis requires testing the equation with localized functions around the concentration points, and these test functions are generally not $\phi$-equivariant, so an estimate restricted to the equivariant subspace $H^1\(\bn\)^\phi$ is sometimes not sufficient. Using the group averaging projection \eqref{Projection map} onto the equivariant subspace, we prove in \Cref{Principle of symetric criticality type result} that the Palais-Smale estimate of minimizing sequence holds for arbitrary test functions in $H^1\(\bn\)$. This result is crucial for our subsequent analysis.\\

We recall that
\begin{align*}
	S^\phi = \inf_{u \in H^1(\bn)^\phi \setminus \{0\}} \f{\norm{u}_\la^2}{\(\int_{\bn} |u|^{2^*} \dv \)^{\f{2}{2^*}}}.
\end{align*}

Now, by Ekeland's variational principle, there exists a minimizing sequence $\{u_n\} \subset H^1\(\bn\)^\phi$ such that
\begin{align*}
	&\text{(a)} \; \int_{\bn} \abs{u_n}^{2^*} \dv = 1, \quad \forall n \in \N,\\
	&\text{(b)} \; \norm{u_n}_\la^2 = S^\phi + o(1), \quad \text{as} \; n \to \infty,\\
	&\text{(c)} \; \< u_n, \varphi \>_\la - S^\phi \int_{\bn} \abs{u_n}^{2^*-2} u_n \varphi \dv = o(1) \norm{\varphi}_\la, \quad \forall \varphi \in H^1\(\bn\)^\phi.
\end{align*}

\medskip

Now, we state the Gromov's covering lemma, which yields a uniformly locally finite covering of $\bn$. For details, see Lemma 1.6 in \cite{Hebey}.\\

\begin{lemma}\label{Gromov's covering lemma}
	Let $(M,g)$ be an $N$-dimensional complete Riemannian manifold with Ricci curvature bounded from below by some $\kappa \in \R$, and let $r>0$ be given. There exists a sequence of points $\{x_j\} \subset M$ such that:\\
	\begin{enumerate}
		\item[(i)] The family $\{B_r(x_j)\}$ is a uniformly locally finite covering of $M$, i.e., $M = \cup_{j \in \N} B_r(x_j)$ and there exists an integer $L$ in terms of $N, r, \kappa$ such that each point $x \in M$ has a neighborhood which intersects at most $L$ many geodesic balls in the family $\{B_r(x_j)\}$;\\
		\item[(ii)] for any $i \neq j$, $ B_{\f{r}{2}}(x_i) \cap B_{\f{r}{2}}(x_j) = \emptyset$.
	\end{enumerate}
\end{lemma}

\medskip

In the next Lemma, we show that up to a subsequence $\{u_n\}$ can be recentered along $\mathrm{Fix}_{\bn} \Ga$, so that there is a fixed amount of positive $L^{2^*}$-mass inside some geodesic ball, uniform with respect to $n$. Also, note that for our groups when $\mathrm{dim} \(\mathrm{Fix}_{\rn} \Ga \) > 0$, we will have $\{(x_1, \cdots, x_N) \in \bn : x_1 = \cdots = x_{N-1} = 0\} \subset \mathrm{Fix}_{\bn} \Ga$. See the examples of groups in the next subsection. \\

\begin{lemma}
	Let $\{u_n\} \subset H^1\(\bn\)^\phi$ such that $\int_{\bn} \abs{u_n}^{2^*} \dv = 1$ and $\norm{u_n}_\la^2 \to S^\phi$, then there exists a sequence of fixed points $\{b_n\} \subset \mathrm{Fix}_{\bn} \Ga$ such that up to a subsequence the functions $v_n = u_n \circ \tau_{b_n}$ satisfy the following:
	\begin{align*}
		\int_{B_{\tilde{R}}} \abs{v_n}^{2^*} \dv \geq \delta,
	\end{align*}
	for some $\tilde{R}, \delta > 0$.
\end{lemma}

\begin{proof}
	Using the coercivity of $Q_\la$, we have that $\{u_n\}$ is bounded in $H^1\(\bn\)$. Let us fix $r>0$, then from the \Cref{Gromov's covering lemma}, we can have a uniformly locally finite covering of $\bn$, i.e., we can choose points $\{x_j : j \in \N\} \subset \bn$ such that $\bn = \cup_{j \in \N} B_r(x_j)$ and there exists a constant $L_r > 0$, depending only on $r, N$ such that
	\begin{align*}
		\sum_{j \in \N} \chi_{B_r(x_j)} \leq L_r, \quad \forall x \in \bn.
	\end{align*}
	From the local Sobolev inequality, we have that
	\begin{align}
		\norm{u_n}_{L^{2^*}\(B_r(x_j)\)} \leq C_r \norm{u_n}_{H^1\(B_{r}(x_j)\)}, \label{local Sobolev inequality}
	\end{align}
	where $C_r > 0$ is a constant, independent of $j$. Now, let us define
	\begin{align*}
		a_{j,n} := \int_{B_r(x_j)} \abs{u_n}^{2^*} \dv \quad \text{and} \quad M_n := \sup_{j} a_{j,n}.
	\end{align*}
	Since the covering is uniformly locally finite, using the fact that $\{u_n\}$ is bounded in $H^1\(\bn\)$, we have
	\begin{align*}
		\sum_{j \in \N} a_{j,n}^{\f{2}{2^*}} \leq C \sum_{j \in \N} \int_{B_r(x_j)} \abs{\nabla_{\bn} u_n}^2 \dv \leq C \int_{\bn} \sum_{j \in \N} \chi_{B_r(x_j)} \abs{\nabla_{\bn} u_n}^2 \dv \leq C \norm{u_n}_\la^2 \leq C,
	\end{align*}
	where the positive constant $C$ is uniform with respect to $n$. This implies
	\begin{align*}
		&1 = \int_{\bn} \abs{u_n}^{2^*} \dv \leq \sum_{j \in \N} a_{j,n} = \sum_{j \in \N} a_{j,n}^{\f{2}{N}} a_{j,n}^{\f{2}{2^*}} \leq M_n^{\f{2}{N}} \cdot C\\
		\Rightarrow & M_n \geq \(\f{1}{C}\)^{\f{N}{2}} =: 2 \delta
	\end{align*}
	Now, for each $n$ we can find a corresponding $j_n$ such that
	\begin{align*}
		a_{j_n,n} = \int_{B_r\(x_{j_n}\)} \abs{u_n}^{2^*} \dv \geq \delta.
	\end{align*}
	Without loss of generality, let us denote the subsequence $\{x_{j_n}\}$ by $\{x_n\}$. We will also use the same notation for further subsequences.\\
	
	If possible, let Case (1) of the \Cref{Dichotomy result} hold. That is, there exists a subsequence $\{x_n\}$ such that, for every $k \in \N$, there exists a subset $\{g_1, \cdots, g_k\}$ of $\Ga$ such that
	\begin{align*}
		d_{\bn} \(g_i x_n, g_j x_n\) \rightarrow &\infty, \; \text{as} \; n \rightarrow \infty, \; \text{for} \; i \neq j.
	\end{align*}
	
	Let us fix a $k \in \N$ such that $k \delta > 1$. Then there exists $M \in \N$ such that for every $n \geq M$,
	\begin{align*}
		& d_{\bn} \(g_i x_n, g_j x_n\) > 2(r+1), \quad \text{for} \; i \neq j \; \text{and} \; i,j \in \{1, \cdots, k\}.\\
	   \text{i.e.,} \; & B_{r+1}\(g_i x_n\) \cap B_{r+1}\(g_j x_n\) = \emptyset, \quad \text{for} \; i \neq j \; \text{and} \; i,j \in \{1, \cdots, k\}
	\end{align*}
	Now, we use the \Cref{change of variable} and the fact $\abs{u_n}^{2^*}$ is $\Ga$-invariant, to get
	\begin{align*}
		1 = \int_{\bn} \abs{u_n}^{2^*} \dv & \geq \sum_{i=1}^{k} \int_{B_r(g_i x_n)} \abs{u_n}^{2^*} \dv\\
		& \geq k \int_{B_r(x_n)} \abs{u_n}^{2^*} \dv \geq k \delta > 1.
	\end{align*}
     This leads to a contradiction. Hence, from \Cref{Dichotomy result} we have a subsequence $\{b_n\} \subset \bn$ such that $\{b_n\} \subset \mathrm{Fix}_{\bn}(\Ga)$ and
	\begin{align}
		d_{\bn} \(b_n, \Ga (x_n)\) & < R, \quad \forall \, n \in \N \label{3.4}.
	\end{align}
	for some $R>0$. Define $v_n (x) := \(u_n \circ \tau_{b_n}\) (x)$, where $\tau_{b_n}$ is the hyperbolic translation. As $\{b_n\} \subset \mathrm{Fix}_{\bn}(\Ga)$, we have
	from \Cref{Equivariant subspaces are closed under translations through the fixed points} that $v_n\in H^1(\bn)^\phi$. Also, from \eqref{3.4} we have that for each $x_n$, there exists $g_n \in \Ga$ such that $ d_{\bn} \(b_n, g_n x_n\) < R + 1, \text{i.e.,} \;\; B_r\(g_nx_n\) \subset B_{R+r+1}\(b_n\).$ Let us take $\tilde{R} = R+r+1$, then
	\begin{align*}
		\int_{B_{\tilde{R}}} \abs{v_n}^{2^*} \dv = \int_{B_{\tilde{R}}(b_n)} \abs{u_n}^{2^*} \dv \geq \int_{B_r\(g_nx_n\)} \abs{u_n}^{2^*} \dv = \int_{B_r(x_n)} \abs{u_n}^{2^*} \dv \geq \delta.
	\end{align*}
	
\end{proof}

\medskip

From the \Cref{change of variable}, we have that $\{v_n\} \subset H^1\(\bn\)^\phi$ and it satisfies the following properties:
\begingroup
\makeatletter\tagsleft@true\makeatother
\begin{align}
	& \int_{\bn} \abs{v_n}^{2^*} \dv = 1, \tag{$\mathbf{M_1}$} \label{2 star norm is 1}\\
	& \int_{B_{\tilde{R}}} \abs{v_n}^{2^*} \dv \geq \delta \tag{$\mathbf{M_2}$} \label{nonvanishing property},\\
	& \norm{v_n}_\la^2 = S^\phi + o(1), \tag{$\mathbf{M_3}$} \label{minimizing sequence}\\
	& \< v_n, \varphi \>_\la - S^\phi \int_{\bn} \abs{v_n}^{2^*-2} v_n \varphi \dv = o(1) \norm{\varphi}_\la, \quad \forall \varphi \in H^1\(\bn\)^\phi. \tag{$\mathbf{M_4}$} \label{subspace PS inequality}
\end{align}

\medskip

\begin{lemma}\label{Principle of symetric criticality type result}
	Let $\{v_n\} \subset H^1\(\bn\)^\phi$ satisfy \eqref{subspace PS inequality}. Then
	\begin{align}
		\< v_n, \psi \>_\la - S^\phi \int_{\bn} \abs{v_n}^{2^*-2} v_n \psi \dv = o(1) \norm{\psi}_\la, \quad \forall \psi \in H^1\(\bn\). \tag{$\mathbf{M_4'}$} \label{fullspace PS inequality}
	\end{align}
\end{lemma}

\endgroup

\begin{proof}
	Define $\Phi : H^1\(\bn\) \to \R$ by
	\begin{align*}
		\Phi(u) = \f{1}{2} \norm{u}_\la^2 - \f{S^\phi}{2^*} \int_{\bn} \abs{u}^{2^*} \dv.
	\end{align*}
	Then \eqref{subspace PS inequality} implies
	\begin{align}
		\abs{\Phi'(v_n) \[\varphi\]} \leq \eps_n \norm{\varphi}_\la, \quad \forall \varphi \in H^1\(\bn\)^\phi,
	\end{align}
	where $\eps_n \to 0$ as $n \to \infty$. Let us now fix $\psi \in H^1\(\bn\)$. We prove this lemma in the following steps:\\
	
	\textbf{Step 1:} We show $\Phi'\(v_n\) \[T_g \psi\] = \Phi'\(v_n\) \[\psi\], \;\; \forall g \in \Ga$.
	
	\medskip
	Since $\{v_n\} \subset H^1\(\bn\)^\phi$, we have $T_g v_n = v_n$, for every $g \in \Ga$. Also, we observe that $\Phi$ is invariant under $T_g$. These imply
	\begin{align*}
		\Phi'\(v_n\) \[T_g \psi\] = \Phi'\(T_g v_n\) \[T_g \psi\] & = \lim\limits_{t \to 0} \f{\Phi \(T_g v_n + t T_g \psi\) - \Phi \(T_g v_n\)}{t}\\
		& = \lim\limits_{t \to 0} \f{\Phi \(T_g \(v_n + t  \psi\)\) - \Phi \(T_g v_n\)}{t}\\
		& = \lim\limits_{t \to 0} \f{\Phi \(v_n + t  \psi\) - \Phi \( v_n\)}{t}\\
		& = \Phi'\(v_n\) \[\psi\].
	\end{align*}
	
	\medskip
	
	\textbf{Step 2:} We show $\norm{P_\phi \psi}_\la \leq \norm{\psi}_\la$.
	
	\medskip
	
	 From \Cref{T is Bochner integrable}, we have that the map $g \mapsto T_g \psi$ is Bochner $\mu$-integrable and then \eqref{triangle inequality for Bochner integral} gives
	\begin{align*}
		\norm{P_\phi \psi}_\la = \norm{\int_{\Ga} T_g \psi \mathrm{~d}\mu(g)}_\la \leq \int_{\Ga} \norm{T_g \psi}_\la \mathrm{~d}\mu(g) = \int_{\Ga} \norm{ \psi}_\la \mathrm{~d}\mu(g) = \norm{ \psi}_\la. 
	\end{align*}
	
	\medskip
	
	\textbf{Step 3:} We show $\Phi'\(v_n\) \[P_\phi \psi\] = \Phi'\(v_n\) \[\psi\]$.
	
	\medskip
	
	Since $\Phi'\(v_n\) \in  \(H^1\(\bn\)\)^*$, from \eqref{Dual elements commute with Bochner integration} we observe
	\begin{align*}
		\Phi'\(v_n\) \[P_\phi \psi\] & = \Phi'\(v_n\) \[\int_{\Ga} T_g \psi \mathrm{~d}\mu(g)\]\\
		& = \int_{\Ga} \Phi'\(v_n\) \[T_g \psi \] \mathrm{~d}\mu(g)\\
		& = \int_{\Ga} \Phi'\(v_n\) \[ \psi \] \mathrm{~d}\mu(g) \quad \[\; \text{Using Step 1} \;\]\\
		& = \Phi'\(v_n\) \[\psi\].
	\end{align*}
	
	\medskip
	
	\textbf{Step 4:} We show $\abs{\Phi'(v_n) \[\psi\]} \leq \eps_n \norm{\psi}_\la, \quad \forall \psi \in H^1\(\bn\)$.
	
	\medskip
	
	Combining Step 3, \eqref{subspace PS inequality}, and Step 2 respectively, we have
	\begin{align*}
		\abs{\Phi'(v_n) \[\psi\]} = \abs{\Phi'\(v_n\) \[P_\phi \psi\]} \leq \eps_n \norm{P_\phi \psi}_\la \leq \eps_n \norm{\psi}_\la.
	\end{align*}
	This completes the proof.
\end{proof}

\begin{theorem}\label{Attainment result when fixed point is nontrivial}
		Let $\(\Ga, \phi\)$ satisfy \eqref{(A_1)}, \eqref{(A_2)} and $ \mathrm{dim}\(\mathrm{Fix}_{\rn} \Ga \) > 0$. Then $S^\phi$ is attained in $H^1\(\bn\)^\phi$.
\end{theorem}

\begin{proof}
	Firstly, we note that $\{v_n\} \subset H^1\(\bn\)^\phi$ is a minimizing sequence for $S^\phi$ such that it has the following properties: \eqref{2 star norm is 1}, \eqref{nonvanishing property}, \eqref{minimizing sequence}, and \eqref{fullspace PS inequality}. From \eqref{minimizing sequence}, we have that $\{v_n\}$ is bounded in $H^1\(\bn\)^\phi$, then up to a subsequence $v_n \rightharpoonup v$ in $H^1\(\bn\)^\phi$.\\
	
	If possible, let $v=0$. Then, there exist two non-negative Radon measures $\al$ and $\beta$ on $\bn$ such that after passing to a subsequence, we have the following convergences in the sense of measure:
	\begin{align*}
		& \abs{\nabla_{\bn} v_n}^2 \rightharpoonup \al, \quad \text{and} \quad \abs{v_n}^{2^*} \rightharpoonup \beta,\\
		\text{i.e.,} & \; \int_{\bn} \psi \abs{\nabla_{\bn} v_n}^2 \dv \to \int_{\bn} \psi \mathrm{~d}\al, \quad \text{and} \quad \int_{\bn} \psi \abs{v_n}^{2^*} \dv \to \int_{\bn} \psi \mathrm{~d}\beta, \;\; \forall \psi \in C_0^\infty\(\bn\).
	\end{align*}
	Let us fix $\eta \in C_c^\infty\(\bn\)$ and take $\psi_n = \eta v_n$. Since $\{v_n\}$ is bounded in $H^1\(\bn\)$, so is $\{\psi_n\}$. Then \eqref{fullspace PS inequality} implies
	\begin{align*}
		& \< v_n, \psi_n \>_\la - S^\phi \int_{\bn} \abs{v_n}^{2^*-2} v_n \psi_n \dv = o(1)\\
		\Rightarrow & \int_{\bn} \eta \abs{\nabla_{\bn} v_n}^2 \dv + \int_{\bn} v_n \< \nabla_{\bn} v_n, \nabla_{\bn} \eta \> \dv - \la \int_{\bn} \eta v_n^2 \dv - S^\phi \int_{\bn} \eta \abs{v_n}^{2^*} \dv = o(1).
	\end{align*}

	Since $v_n \rightharpoonup 0$ in $H^1\(\bn\)$, then by Rellich-Kondrachov theorem we have $v_n \to 0$ in $L^2_{loc} \(\bn\)$. Now combining this, Cauchy-Schwarz inequality, and the fact $\{v_n\}$ is bounded in $H^1\(\bn\)$, we have
	\begin{align*}
		\int_{\bn} v_n \< \nabla_{\bn} v_n, \nabla_{\bn} \eta \> \dv = o(1), \quad \text{and} \quad \int_{\bn} \eta v_n^2 \dv = o(1).
	\end{align*}
	Hence, we have
	\begin{align*}
		\int_{\bn} \eta \abs{\nabla_{\bn} v_n}^2 \dv - S^\phi \int_{\bn} \eta \abs{v_n}^{2^*} \dv = o(1).
	\end{align*}
	Now, letting $n \to \infty$ we observe
	\begin{align}
		\int_{\bn} \eta \mathrm{~d}\al = S^\phi \int_{\bn} \eta \mathrm{~d}\beta, \;\; \forall \eta \in C_c^\infty\(\bn\). \label{Measure equality in dual of compat support functions}
	\end{align}
	Note that $C_c^\infty\(\bn\)$ is dense in $C_0^\infty\(\bn\)$ with respect to the $L^\infty$-norm. Using this density, we can show that the above equality is true for every $\psi \in C_0^{\infty}\(\bn\)$. Hence,
	\begin{align}
		\al = S^\phi \beta
	\end{align}
	Now from the sharp Poincar\'e-Sobolev inequality \eqref{Poincare-Sobolev} and the fact $S_{\la, 2^*-1} = S$, when $\la = \f{N(N-2)}{4}$, we have the sharp Sobolev inequality on $\bn$:
	\begin{align}
		S \(\int_{\bn} \abs{f}^{2^*} \dv\)^{\f{2}{2^*}} \leq \int_{\bn} \( \abs{\nabla_{\bn} f}^2 - \f{N(N-2)}{4} f^2\) \dv, \quad \forall f 
		\in H^1\(\bn\). \label{Sobolev inequality on hyperbolic space}
	\end{align} 
	Let us take $f = \eta v_n$. Now, again using $v_n \to 0$ in $L^2_{loc}\(\bn\)$ the above inequality becomes
	\begin{align*}
		S \(\int_{\bn} \abs{\eta}^{2^*} \mathrm{~d}\beta\)^{\f{2}{2^*}} \leq \int_{\bn} \eta^2 \mathrm{~d}\al.
	\end{align*}
	Since $\al = S^\phi \beta$, we get
	\begin{align}
		S \(\int_{\bn} \abs{\eta}^{2^*} \mathrm{~d}\beta\)^{\f{2}{2^*}} \leq S^\phi \int_{\bn} \eta^2 \mathrm{~d}\beta, \quad \forall \eta \in C_c^\infty\(\bn\). \label{inequality with S and S^phi}
	\end{align}
	Let $E$ be a relatively compact Borel subset of $\bn$. By regularity of the Radon measure, we can approximate the function $\chi_E$ by a sequence of smooth compactly supported functions $0 \leq \eta_i \leq 1$. Each $\eta_i$ satisfies the above inequality, now passing to the limit as $i \to \infty$ in \eqref{inequality with S and S^phi} we have $S \beta \(E\)^{\f{2}{2^*}} \leq S^\phi \beta \(E\)$. This implies
	\begin{align*}
		\text{either} \; \beta (E) = 0, \; \text{or} \; \beta \(E\) \geq \(\f{S}{S^\phi}\)^{\f{N}{2}} > 0.
	\end{align*}
	Also, we have $\beta \(\bn\) \leq 1$, because $\int_{\bn} \abs{v_n}^{2^*} \dv = 1$. If possible, let $\beta$ have a positive component, which is not a Dirac mass. Then for every open set $E \subset \bn$, we will have $\beta \(E\) \geq \(\f{S}{S^\phi}\)^{\f{N}{2}}$, that will contradict the fact $\beta$ is a finite measure. Hence, $\beta$ is purely a sum of Dirac masses. Hence there exists a finite index set $\digamma$, a family $\{x_j : j \in \digamma\}$ of distinct points in $\bn$ and a family $\{\beta_j : j \in \digamma\}$ of positive numbers such that $\beta = \sum_{j \in \digamma} \beta_j \delta_{x_j}$. Furthermore, we will have $\# \digamma \leq \(\f{S^\phi}{S}\)^{\f{N}{2}}$. Similarly, for the measure $\al$ we will have $\al = \sum_{j \in \digamma} \al_j \delta_{x_j},$ where $\al_j = S^\phi \beta_j, \; \forall j \in \digamma$.\\
	
	Let $\xi \in C_c^\infty\(\bn\)$ such that $0 \leq \xi \leq 1$ and $\xi \equiv 1$ on $B_{\tilde{R}}$, then from \eqref{nonvanishing property}, we obtain
	\begin{align*}
		\int_{\bn} \xi \abs{v_n}^{2^*} \dv \geq \int_{B_{\tilde{R}}} \abs{v_n}^{2^*} \dv \geq \delta \;
		\Rightarrow \; \int_{\bn} \xi \mathrm{~d} \beta \geq \delta.
	\end{align*}
	This implies the Radon measure $\beta$ is not zero, i.e., one of $\beta_j, \; j \in \digamma$ is non-zero. Also, we have that $\abs{v_n}^{2^*}$ is $\Ga$-invariant, i.e., the measure $\abs{v_n}^{2^*} \dv$ is $\Ga$-invariant. Then we can show that $\beta$ is also $\Ga$-invariant, i.e., $\beta \(gE\) = \beta (E)$, for every Borel subset $E$ in $\bn$ and $g \in \Ga$. Now, we observe that for every $x \in \bn \setminus \mathrm{Fix}_{\bn} \Ga$, the orbit $\Ga(x)$ will have infinitely many points. This implies the points $x_j \in \mathrm{Fix}_{\bn} \Ga, \; \forall j \in \digamma$.\\
	
	Let $x_0$ be one of the points in the set $\{x_j : j \in \digamma\}$. Let us choose $\eta \in C_c^\infty \([0, \infty)\)$ such that $0 \leq \eta \leq 1$ and
	\begin{align*}
		\eta (t) = \begin{cases}
			1, \quad \text{when} \; t < \f{1}{2}\\
			0, \quad \text{when} \; t > 1.
		\end{cases}
	\end{align*}
	Now, for any $x \in \bn$ and $\eps > 0$, we define
	\begin{align*}
		\eta_\eps (x) = \eta \(\f{d_{\bn}(x,x_0)}{\eps}\).
	\end{align*} 
	Furthermore, we define functions $f_{n,\eps} : \rn \to \R$ such that
	\begin{align*}
		f_{n, \eps} (y) = \begin{cases}
			\eta_\eps \(\mathrm{exp}_{x_0} y\) v_n \(\mathrm{exp}_{x_0} y\), \quad & |y| < \eps;\\
			0, & |y| \geq \eps.
		\end{cases}
	\end{align*}
	
	\medskip
	
	\textbf{Claim:} $f_{n,\eps} \in D^{1,2}\(\rn\)^\phi$.
	
	Firstly, we recall that if $M$ is a connected, complete Riemannian manifold and $F : M \to M$ is an isometry, then
	\begin{align*}
		F \(\mathrm{exp}_x (\nu)\) = \mathrm{exp}_{F(x)} \(\mathrm{d}F_x \nu\), \label{relationship between isometry and exponential}
	\end{align*}
	for all $x \in M$ and $ \nu \in T_xM$. Note that every $g \in \Ga$ is an isometry on $\bn$, therefore the above relationship becomes
	\begin{align}
		g \(\mathrm{exp}_{x_0} \(y\)\) = \mathrm{exp}_{gx_0} \(gy\) = \mathrm{exp}_{x_0} \(gy\), \;\; \forall y \in \rn \cong T_{x_0} \bn,
	\end{align}
	since $x_0 \in \mathrm{Fix}_{\bn}\Ga$, i.e., $gx_0 = x_0$. Now, in $\{|y| < \eps\}$ we have
	\begin{align*}
		f_{n,\eps} (gy) = \eta_{\eps} \(\mathrm{exp}_{x_0} gy\) v_n \(\mathrm{exp}_{x_0} gy\) & = \eta_{\eps} \(g \(\mathrm{exp}_{x_0} y\)\) v_n \(g \(\mathrm{exp}_{x_0} y\)\)\\
		& = \eta_\eps \(\mathrm{exp}_{x_0} y\) \cdot \phi(g) v_n \(\mathrm{exp}_{x_0} y\) = \phi (g) f_{n, \eps} (y).
	\end{align*}
	Hence, $f_{n,\eps}$ is $\phi$-equivariant. Since $f_{n,\eps}$ has compact support, we have $f_{n,\eps} \in D^{1,2} \(\rn\)^\phi$. This proves the claim.\\
	
	Now, from the definition of $S_\infty^\phi$, we get
		\begin{align*}
		S_\infty^\phi \(\int_{\rn} \abs{f_{n,\eps}}^{2^*} \mathrm{~d}y\)^{\f{2}{2^*}} \leq \int_{\rn} \abs{\nabla f_{n,\eps}}^2 \mathrm{~d}y.
	\end{align*}
	From Lemma 2.24 from the book by T. Aubin \cite{Aubin_book}, we have
	\begin{align*}
		S_{\infty, \eps}^\phi \(\int_{\bn} \abs{\eta_\eps v_n}^{2^*} \dv\)^{\f{2}{2^*}} \leq \int_{\bn} \abs{\nabla_{\bn} \(\eta_{\eps} v_n\)}^2 \dv
	\end{align*}
    $S_{\infty, \eps}^\phi$ does not depend on the choice of $x_0 \in \bn$ and as $\eps \to 0$, we have $S_{\infty, \eps}^\phi \to S_\infty^\phi$. Again, using the $L^2_{loc}$ convergence of $v_n$ as before and after that taking $\eps \to 0$, we obtain 
	\begin{align*}
		S_\infty^\phi \beta_0^{\f{2}{2^*}} \leq \al_0 \;
		\Rightarrow \; S_\infty^\phi \beta_0^{\f{2}{2^*}} \leq S^\phi \beta_0 \;
		\Rightarrow \; \beta_0 \geq \(\f{S_\infty^\phi}{S^\phi}\)^{\f{N}{2}},
	\end{align*}
	where $\al_0 = \al \(\{x_0\}\)$ and $\beta_0 = \beta \(\{x_0\}\)$. Now, from the \Cref{local compactness level} we have $S^\phi < S_\infty^\phi$, i.e., we will get $\beta_0 > 1$, which is impossible. Hence $v \neq 0$.\\
	
	Let $w_n = v_n - v$ and let us denote $m := \int_{\bn} \abs{v}^{2^*} \dv \in (0,1]$. Then, from the Brezis-Lieb Lemma we have
	\begin{align}
		1 = \int_{\bn} \abs{v_n}^{2^*} \dv = m + \int_{\bn} \abs{w_n}^{2^*} \dv + o(1). \label{Brezis-Lieb}
	\end{align}
	From the weak convergence $v_n \rightharpoonup v$ in $H^1\(\bn\)$, we obtain
	\begin{align}
		\norm{v_n}_\la^2 = \norm{v}_\la^2 + \norm{w_n}_\la^2 + o(1). \label{Brezis-Lieb for norm}
	\end{align}
	Now, from the definition of $S^\phi$, we get
	\begin{align*}
		S^\phi \(\int_{\bn} \abs{v}^{2^*} \dv\)^{\f{2}{2^*}} \leq \norm{v}_\la^2, \; \text{and} \; S^\phi \(\int_{\bn} \abs{w_n}^{2^*} \dv\)^{\f{2}{2^*}} \leq \norm{w_n}_\la^2.
	\end{align*}
	Now combining these with \eqref{Brezis-Lieb} and \eqref{Brezis-Lieb for norm}, we have
	\begin{align*}
		S^\phi \geq S^\phi \(m^{\f{2}{2^*}} + (1-m)^{\f{2}{2^*}}\). 
	\end{align*}
	We observe that $m \mapsto m^{\f{2}{2^*}} + (1-m)^{\f{2}{2^*}}$ is strictly concave in $[0,1]$. This implies 
	\begin{align*}
		m^{\f{2}{2^*}} + (1-m)^{\f{2}{2^*}} > 1, \; \forall m \in (0,1).
	\end{align*}
	Hence, only possibility is $m = 1$, i.e., $\int_{\bn} \abs{v}^{2^*} \dv = 1$. Therefore, from the weak lower semi-continuity of norm $\norm{\cdot}_\la$, we conclude that $v$ is a minimizer of $S^\phi$.
	
\end{proof}

\subsection{Conclusion}

\textbf{Example of groups and corresponding $\phi$ satisfying \eqref{(A_1)} and \eqref{(A_2)}:}\\

	Let us first define a group action on $\R^4$. Let us set $z = (z_1, z_2) \in \R^4$ and for $\theta \in [0.2\pi)$, let us take the rotation matrix to be 
\begin{align*}
	R_\theta := \begin{bmatrix}
		\cos\theta & -\sin\theta\\
		\sin\theta & \cos\theta
	\end{bmatrix}
	\in SO(2),
\end{align*}
then the action on $\R^4$ is given by $r_\theta (z) := \(R_\theta z_1, R_\theta z_2\)$. Also, let us consider an action to be
\begin{align*}
	\tau \(z_1, z_2\) := \(Az_1, B z_2\), \quad \text{where} \;\; A:= \begin{bmatrix}
		-1&0\\
		0&1
	\end{bmatrix},
	\quad
	B:= \begin{bmatrix}
		1&0\\
		0&-1
	\end{bmatrix}.
\end{align*}
Let $\Sigma$ be the group generated by $\{\tau, r_\theta : \theta \in [0,2\pi)\}$ and a corresponding onto homomorphism $\phi : \Sigma \to \mathbb{Z}_2$ such that
\begin{align*}
	\phi \(r_\theta\) = 1, \quad \text{and} \quad \phi (\tau) = -1.
\end{align*}

Now fix $j \in \{1, 2, \cdots, n\}$ and identify $\rn$ to be $\(\R^4\)^j \times \R^{N-4j}$ and $x = (z_1, \cdots, z_j, y) \in \rn$. Further, set
\begin{align*}
	\Lambda_j = \begin{cases}
		O(N-4j), \quad & \text{if} \; j \neq n\\
		\{I\} & \text{if} \; j = n
	\end{cases}
\end{align*}
Define the group $\Ga_j := \Sigma^j \times \Lambda_j \subset_{\text{closed}} O(N)$ and its action of $\rn$ is given by
\begin{align}
	(\sigma_1, \cdots, \sigma_j, \eta ) (z_1, \cdots, z_j, y) := (\sigma_1 z_1, \cdots, \sigma_j z_j, \eta y), \label{Group examples}
\end{align}
where $(\sigma_1, \cdots, \sigma_j) \in \Sigma \times \cdots \times \Sigma \; (j \; \text{times})$. Moreover, $\phi_j : \Ga_j \to \mathbb{Z}_2$ is given by
\begin{align}
	\phi_j (\sigma_1, \cdots, \sigma_j, \eta ) := \phi (\sigma_1) \cdots \phi (\sigma_j). \label{Homomorhisms example}
\end{align}
It is easy to see that each $\phi_j$ is a continuous onto homomorphism. Furthermore, the conditions \eqref{(A_1)} and \eqref{(A_2)} hold. Also, we observe that
\begin{align*}
	\mathrm{Fix}_{\rn} \Ga_j = \begin{cases}
		\{0\}, \quad & \text{if} \; j \neq n,\\
		\R^{N-4j}, & \text{if} \; j = n.
	\end{cases}
\end{align*}
If $j=n$, we will have $\mathrm{Fix}_{\bn} \Ga_j = \{(x_1, \cdots, x_N) \in \bn : x_1 = \cdots = x_{4j} = 0\}$. Also, we observe that $D^{1,2}\(\rn\)^{\phi_i} \cap D^{1,2}\(\rn\)^{\phi_j} = \{0\}$, for all $i, j \in \{1, \cdots, n\}$ and $i \neq j$. For detailed proof, see Lemma 3.2 in \cite{Clapp_Rios}. Consequently, we will have $H^1\(\bn\)^{\phi_i} \cap H^1\(\bn\)^{\phi_j} = \{0\}$ and $H_0^1\(\Bee\)^{\phi_i} \cap H_0^1\(\Bee\)^{\phi_j} = \{0\}$.\\

\begin{remark}
	For a better understanding of the group actions defined in \eqref{Group examples}, we present the groups acting on $\R^{12}$ and $\R^{13}$:
	\begin{enumerate}
		\item[1.] For $R^{12}$ the groups are
		\begin{enumerate}
			\item[a.] $\Sigma \times O(8)$ acting on $\R^{12} = \R^4 \times \R^8$.
			\item[b.] $\Sigma \times \Sigma \times O(4)$ acting on $\R^{12} = \R^4 \times \R^4 \times \R^4$.
			\item[c.] $\Sigma \times \Sigma \times \Sigma$ acting on $\R^{12} = \R^4 \times \R^4 \times \R^4$.
		\end{enumerate}
		\item[2.] For $R^{13}$ the groups are
		\begin{enumerate}
			\item[a.] $\Sigma \times O(9)$ acting on $\R^{13} = \R^4 \times \R^9$.
			\item[b.] $\Sigma \times \Sigma \times O(5)$ acting on $\R^{13} = \R^4 \times \R^4 \times \R^5$.
			\item[c.] $\Sigma \times \Sigma \times \Sigma \times \{I\}$ acting on $\R^{13} = \R^4 \times \R^4 \times \R^4 \times \R$.
		\end{enumerate}
	\end{enumerate}
	
\end{remark}

\medskip

\begin{proof}[Proof of \Cref{Existence theorem}]
	Let $N = 4n +m$, where $ n \geq 1$ and $m \in \{0,1,2,3\}$. Let $\(\Ga_j, \phi_j\)$, $j \in \{1, \cdots, n\}$ be as in \eqref{Group examples} and \eqref{Homomorhisms example}. Then from the \Cref{Corollary for existence in the trivial fixed point case} and the \Cref{Attainment result when fixed point is nontrivial}, we argue that for each $j \in \{1, \cdots, n\}$, there is a minimizer $w_j \in H^1\(\bn\)^{\phi_j}$ for $S^{\phi_j}$. Now, by Lagrange's multiplier theorem, we obtain
	\begin{align*}
		\< w_j, \varphi \>_\la = S^{\phi_j} \int_{\bn} \abs{w_j}^{2^*-2} w_j \varphi \dv, \quad \forall \varphi \in H^1\(\bn\)^{\phi_j}.
	\end{align*}
	Taking the transformation $\tilde{w}_j = \(S^{\phi_j}\)^{\f{N-2}{4}} w_j$ gives
	\begin{align*}
		\< \tilde{w}_j, \varphi \>_\la = \int_{\bn} \abs{\tilde{w}_j}^{2^*-2} \tilde{w}_j \varphi \dv, \quad \forall \varphi \in H^1\(\bn\)^{\phi_j}.
	\end{align*}
	Now by the principle of symmetric criticality \cite{Kobayashi_Otani}, we get
	\begin{align*}
		\< \tilde{w}_j, \psi \>_\la = \int_{\bn} \abs{\tilde{w}_j}^{2^*-2} \tilde{w}_j \psi \dv, \quad \forall \psi \in H^1\(\bn\).
	\end{align*}
    Therefore $\tilde{w}_j \in H^1\(\bn\)^{\phi_j}$ is a nontrivial solution of \eqref{(1)}, when $p = 2^*-1$. We recall that each nontrivial function in equivariant subspaces is nonradial and sign-changing. Also, we have $H^1\(\bn\)^{\phi_i} \cap H^1\(\bn\)^{\phi_j} = \{0\}$, for all $i, j \in \{1, \cdots, n\}$ and $i \neq j$. Hence, the theorem.
\end{proof}

\section{Non-existence result for $\la > \la_1$}\label{non-existence section}
Throughout this section, we assume that $\la > \la_1, \; 1<p\leq 2^*-1$ and prove the non-existence \Cref{Non-existence theorem for higher eigenvalues}. The condition $(QOP)_{\Ga, d}$, as in \Cref{orbit packing number}, provides sufficiently many well-separated points along each $\Ga$-orbit. Combining this with the regularity of the solution, we obtain a sufficient decay estimate to conclude our result, using the spectral non-existence criterion \eqref{Marzuola and Borthwick result}. We begin with a standard regularity result as follows:\\

\begin{lemma}\label{u goes to 0}
	Let $u$ be a solution of \eqref{(1)}, with $ \la \in \R, \;1 < p \leq 2^* - 1$. Then $u, \nabla_{\bn}u \in L^\infty(\bn)$ and $u(x) \to 0$ and $\abs{\nabla_{\bn} u(x)} \to 0$ as $d_{\bn}(0,x) \to \infty$. 
\end{lemma}

\medskip

The proof follows from the classical Br\'ezis--Kato \cite{Struwe} results and the Moser iteration technique. See \cite{Bhakta_Sandeep} for a proof in the context of hyperbolic space; although it is stated for $\la < \la_1$, the argument works for any $\la \in \R$. However, this does not provide a quantitative decay estimate for nontrivial solutions to \eqref{(1)}, that we can use for \eqref{Marzuola and Borthwick result}.

Now we shall describe the orbit packing number for the orbits in $\bn$. The lower bound of this number is deduced from the $(QOP)_{\Ga,d}$ condition in \Cref{orbit packing number} as follows:  

Let us recall that $x = (s, \theta)$ in the Euclidean polar coordinate, i.e., $|x| = s$ and $\theta \in \Sl^{N-1}$ and $x = (\rho, \theta)$ is in terms of the geodesic polar coordinate in $\bn$, where $\rho = d_{\bn} (0,x)$ and $s = \tanh \f{\rho}{2}$. We will fix the notation throughout.\\

\begin{definition}
	Let $\Ga \subset_{\text{closed}} O(N)$. For $x \in \bn$ and $\zeta > 0$, define
	\begin{align*}
		m_\Ga (x, \zeta) & := \sup \{m \in \N \ : \ \exists \ g_1, \cdots, g_m \in \Ga \ \text{such that} \ d_{\bn} \(g_i x, g_j x\) \geq 2\zeta, \ \text{for} \ i \neq j\}.
	\end{align*}
\end{definition}

\begin{lemma}
	Let $\Ga \subset O(N)$ satisfy the condition $(QOP)_{\Ga, d}$. Then there exists $\rho_0, C, \zeta_0 >0$ such that
	\begin{align*}
		m_\Ga (x, \zeta) \geq C \(\f{\sinh \rho}{\zeta}\)^d, 
	\end{align*}
	for every $x = \(\rho, \theta\)$ with $\rho \geq \rho_0$ and $0 < \zeta < \zeta_0$.
\end{lemma}

\begin{proof}
	Suppose $\theta_1, \cdots, \theta_m \in \Ga \theta$ such that $\abs{\theta_i - \theta_j} \geq \eta$, for all $i,j \in \{1, \cdots, m\}$ and $i \neq j$. Let us define the corresponding points on the orbit $\Ga(x)$ to be $x_i = (s, \theta_i)$ for $i \in \{1, \cdots, m\}$. Now for $i \neq j$
	\begin{align*}
		\abs{x_i - x_j} = s \abs{\theta_i - \theta_j} \geq s \eta.
	\end{align*}
	From \eqref{hyperbolic metric}, we have
	\begin{align*}
		d_{\bn} (x_i, x_j) & = \cosh^{-1} \( 1+ \f{2\abs{x_i - x_j}^2}{(1-s^2)^2}\)\\
		& = 2 \sinh^{-1} \(\f{\abs{x_i - x_j}}{1-s^2}\) \qquad \[\; \cosh \(2 \sinh^{-1} a\) = 1 + 2a^2 \; \]\\
		& \geq 2 \sinh^{-1} \(\f{s\eta}{1-s^2}\) = 2 \sinh^{-1} \(\f{\eta}{2} \sinh \rho\).
	\end{align*}
	Since $\Ga$ satisfies the condition $(QOP)_{\Ga, d}$, there exist constants $C_0 >0, \ \eta_0 > 0$ such that
	\begin{align*}
		M_\Ga (\theta, \eta) \geq C_0 \eta^{-d}, \;\; \forall \eta \in (0,\eta_0).
	\end{align*}
	Let us fix $\rho_0 > 0$ very large. Then there exists $\zeta_0 > 0$ depending on $\eta_0$ and a constant $C>0$ such that
$ \f{2 \sinh \zeta}{\sinh \rho} < \eta_0, \;\; \forall \rho > \rho_0, \ \text{and} \ \zeta \in  (0, \zeta_0),$ and $\sinh \zeta \leq C \zeta, \;\; \forall \zeta \in  (0, \zeta_0).$ Now, take $\eta = \f{2 \sinh \zeta}{\sinh \rho}$, then $d_{\bn} (x_i, x_j) \geq 2 \zeta.$ Therefore for $x = (\rho, \theta)$, with $\rho > \rho_0$ we have
	\begin{align*}
		m_\Ga \(x, \zeta\) \geq M_{\Ga} \(\theta, \eta\) \geq C_0 \(\f{\sinh \rho}{2 \sinh \zeta}\)^d \geq C \(\f{\sinh \rho}{\zeta}\)^d, \;\; \forall  \zeta \in  (0, \zeta_0).
	\end{align*}
\end{proof}

\begin{remark}
	The condition $(QOP)_{\Ga, d}$ is closely related to the orbit expansion arising in compactness results for invariant Sobolev spaces; see \cite{Farkas_Kristaly_Mester} and \cite{Manna_Manna1}. In these articles, the geometric requirement for the compact embedding is qualitative, i.e., for a fixed radius $\zeta > 0$, the maximal number $m_{\Ga}(x,\zeta)$ of mutually disjoint geodesic balls of radius $\zeta$, with centers at the orbit $\Ga x$ is required to satisfy:
	\begin{align*}
		m_{\Ga} (x, \zeta) \to \infty, \quad \text{as} \; d_{\bn}(0,x) \to \infty.
	\end{align*}
	The condition $(QOP)_{\Ga, d}$ contains additional quantitative information. By the previous lemma, we have
	\begin{align*}
		m_\Ga (x, \zeta) \geq C \(\f{\sinh \rho}{\zeta}\)^d, \quad x = (\rho, \theta),
	\end{align*}
	for $\rho$ sufficiently large and $\zeta >0$ sufficiently small. Thus, for every fixed $\zeta >0$,
	\begin{align*}
		m_\Ga (x, \zeta) \gtrsim \(\sinh \rho\)^d \approx e^{\rho d}, \quad \text{as} \; \rho \to \infty.
	\end{align*}
	The parameter $d$ measures the rate of orbit expansion at infinity of hyperbolic space.
\end{remark}

\vspace*{5mm}

\begin{lemma}\label{Decay for solutions}
	Let $u$ solve \eqref{(1)} and $\abs{u}$ is $\Ga$-invariant, where $\Ga \subset O(N)$ satisfies the condition $(QOP)_{\Ga, d}$ for some $d>0$. Then 
	\begin{align*}
		|u(x)| = \abs{u(\rho, \theta)} \leq C \( \sinh \rho\)^{- \f{d}{N-d+2}}, \qquad |x| \nearrow 1. 
	\end{align*}
\end{lemma}

\begin{proof}
	Since $u$ is a solution to \eqref{(1)}, from \Cref{u goes to 0} we have $u, \nabla_{\bn} u \in L^{\infty}\(\bn\)$. From the elliptic regularity, we have $u \in C^{1,\al} \(\bn\), \ \al > 0$. For a point $x = (\rho, \theta) \in \bn$, we denote
	\begin{align*}
		a := \abs{u(x)}, \quad \text{and} \; L := \norm{\nabla_{\bn}u}_{L^{\infty}\(\bn\)}.
	\end{align*}
	If $y$ is close to $x$, then from the mean value theorem we have 
	\begin{align*}
		\abs{u(y) - u(x)} \leq L \ d_{\bn}(x,y).
	\end{align*}
	Let us denote $\zeta := \f{a}{4L}$. Now, we can choose $x$ large enough, i.e., $\zeta$ small enough so that 
	\begin{align*}
		y \in B_{\zeta}(x) & \Rightarrow \ \abs{u(y) - u(x)} \leq \f{a}{4} \Rightarrow \ |u(y)| \geq |u(x)| - \f{a}{4} = \f{3a}{4} \geq \f{a}{2}.
	\end{align*}
	By the invariance, the same lower bound holds on all balls of radius $\zeta$ centered at points of the orbit $\Ga x$. Now, we can invoke the previous Lemma. Let us choose $\zeta < \zeta_0$ small, then
	\begin{align*}
		m_{\Ga} \(x, \zeta\) \geq C \(\f{\sinh \rho}{\zeta}\)^d.
	\end{align*}
	Let us recall that there are $m_{\Ga} \(x, \zeta\)$ many mutually disjoint geodesic balls of radius $\zeta$, all centered on $\Ga x$, and on each of these balls we have $|u| \geq \f{a}{2}$. Therefore
	\begin{align*}
		& \int_{\bn} |u|^2 \dv \geq m_{\Ga} \(x, \zeta\) \cdot a^2 \ \mathrm{Vol}_{\bn} \(B_\zeta\) \geq C \(\f{\sinh \rho}{\zeta}\)^d \cdot a^2 \zeta^N = C \ (\sinh \rho)^d \ a^{N-d+2}\\
		\Rightarrow & \ |u(x)| = a \leq C \(\sinh \rho\)^{-\f{d}{N-d+2}}.
	\end{align*}
	This completes the proof.
\end{proof}

\vspace*{2mm}

\begin{proof}[Proof of \Cref{Non-existence theorem for higher eigenvalues}]
	From the \Cref{u goes to 0}, we have $V \in L^{\infty}(\bn)$ and the previous lemma gives
	\begin{align*}
		\abs{V} = |u(x)|^{p-1} \leq C \(\sinh \rho\)^{-\f{d(p-1)}{N-d+2}} = o\(\rho^{-1}\), \;\; \text{as} \; |x| \nearrow 1.
	\end{align*}
	Hence, \eqref{Marzuola and Borthwick result} concludes the proof.
\end{proof}

\subsection{Example of groups satisfying $(QOP)_{\Ga, d}$}

Here we give a few examples of closed subgroups of $O(N)$ that satisfy or do not satisfy the condition $(QOP)_{\Ga, d}$.\\

\begin{proposition}
	 $\Ga = O(N)$ satisfies $(QOP)_{\Ga, N-1}$.
\end{proposition}

\begin{proof}
	In this case, for any $\theta \in \Sl^{N-1}$, the orbit $\Ga \theta = \Sl^{N-1}$. Let us take $\eta_0 = 1$ and $\eta \in (0,\eta_0)$. Let $K$ be the maximal number of points $\theta_1, \cdots, \theta_K \in \Sl^{N-1}$ such that 
	\begin{align*}
		\abs{\theta_i - \theta_j} \geq \eta, \quad \forall i, j \in \{1, \cdots, K\} \; \text{and}, \; i \neq j.
	\end{align*} 
	Therefore, $B_S (\theta_i, \eta) := B \(\theta_i, \eta\) \cap \Sl^{N-1}, \ i \in \{1, \cdots, K\}$ covers the entire $\Sl^{N-1}$. Let us denote $\sigma$ to be the surface measure, then
	\begin{align*}
		& \sigma \(\Sl^{N-1}\) \leq \sum_{i=1}^{K} \sigma \(B_S (\theta_i, \eta)\) \leq K \(C_N \eta^{N-1}\) \;
		\Rightarrow \; K \geq \f{\sigma \(\Sl^{N-1}\)}{C_N} \eta^{-(N-1)}.
	\end{align*} 
	Hence, the proof.
\end{proof}

\begin{proposition}
	Let $\rn = \R^{k_1} \times \R^{k_2} \times \cdots \times \R^{k_l}, \; \Ga = O(k_1) \times O(k_2) \times \cdots \times O(k_l)$ and $k_* = \min_{1 \leq i \leq l} k_i \geq 2$, then $\Ga$ satisfies $(QOP)_{\Ga, k_*-1}$.
\end{proposition}

\begin{proof}
	Let $a \in (0,\infty)$ and fix $\eta \in (0,a]$. Let us choose a maximal family of points $\xi_1, \cdots, \xi_M \in  \Sl_a^{k_*-1}$ such that $\abs{\xi_i - \xi_j} \geq \eta$, for $i \neq j$. From the maximality of the set, it is clear that for every $\xi \in  \Sl_a^{k_*-1}$, there exists $i \in \{1, \cdots, M\}$, such that $\abs{\xi - \xi_i} < \eta$. Therefore the sets
	\begin{align*}
		U_i := \{\xi \in \Sl_a^{k_*-1} \; : \; \abs{\xi - \xi_i} < \eta\}, \quad i \in \{1, \cdots, M\}
	\end{align*}
	cover the sphere $\Sl_a^{k_*-1}$. This implies
	\begin{align}
		\mathcal{H}^{k_*-1} \( \Sl_a^{k_*-1} \) \leq \sum_{i=1}^{M} \mathcal{H}^{k_*-1} \( U_i \). \label{cap estimate}
	\end{align}
	
	\textbf{Claim:} $ \mathcal{H}^{k_*-1} \( U_i \) \leq C \eta^{k_*-1}$, for some constant $C \equiv C(k_*)$.\\
	
	Since the sphere is rotationally symmetric, it is enough to estimate the cap centered at $\xi_0 := a e_{k_*}$, where $e_{k_*} = \(0, \cdots, 0, 1, 0, \cdots, 0\)$. Consider the set
	\begin{align*}
		U := \{\xi \in \Sl_a^{k_*-1} \; : \; \abs{\xi - \xi_0} < \eta\}.
	\end{align*}
	Let us write $\xi = a \omega$, where $\omega \in \Sl^{k_*-1}$ and $\theta \in [0,\pi]$ be the angle between $\omega$ and $e_{k_*}$. Then
	\begin{align*}
		 \abs{\xi - \xi_0} = a \abs{\omega - e_{k_*}} = 2a \sin \f{\theta}{2} < \eta \;
		\Rightarrow \; \sin \f{\theta}{2} < \f{\eta}{2a} \;
		\Rightarrow \; \theta < 2 \sin^{-1} \(\f{\eta}{2a}\).
	\end{align*}  
	For $ t \in [0, \f{1}{2}]$, we observe $ \sin^{-1} t \leq 2t$. Since $\f{\eta}{2a} \leq \f{1}{2}$, we have $\theta < \f{2\eta}{a}$. Hence
	\begin{align*}
		U \subset V := \{\xi = a \omega \; : \; \mathrm{angle} (\omega, e_{k_*}) < \f{2\eta}{a}\}
	\end{align*}
	This implies
	\begin{align*}
		\mathcal{H}^{k_*-1} \(V\) & = a^{k_*-1} \sigma_{k_*-2} \int_{0}^{\f{2\eta}{a}} \(\sin s\)^{k_*-2} \mathrm{d}s\\
		& \leq a^{k_*-1} \sigma_{k_*-2} \int_{0}^{\f{2\eta}{a}} s^{k_*-2} \mathrm{d}s\\
		& = a^{k_*-1} \sigma_{k_*-2} \f{1}{k_*-1} \(\f{2\eta}{a}\)^{k_*-1} = \f{2^{k_*-1} \sigma_{k_*-2}}{k_*-1} \eta^{k_*-1},
	\end{align*}
	where $\sigma_{k_*-2}$ is the surface measure of $\Sl^{k_*-2}$. For every $i \in \{1, \cdots, M\}$, we obtain
	\begin{align*}
		\mathcal{H}^{k_*-1} (U_i) \leq \mathcal{H}^{k_*-1} (V) \leq C \eta^{k_*-1},
	\end{align*}
	where $C$ is a positive constant, depending only on $k_*$. This completes the proof of the claim. Now from the \eqref{cap estimate}, we have
	\begin{align}
		 \mathcal{H}^{k_*-1} \( \Sl_a^{k_*-1} \) =\sigma_{k_*-1} a^{k_*-1} \leq M \(C \eta^{k_*-1}\) \;
		\Rightarrow \; M \geq C \(\f{a}{\eta}\)^{k_*-1}. \label{Lower bound for the minimal sphere}
	\end{align}
	
	Let us write $\theta \in \Sl^{N-1}$ to be
	\begin{align*}
		\theta = \(\theta_1, \cdots, \theta_l\), \; \; \text{where} \; \theta_i \in \R^{k_i}, \; \text{and} \; \sum_{i=1}^{l} \abs{\theta_i}^2 = 1.
	\end{align*}
	There exists $j \in \{1, \cdots, l\}$ such that $\abs{\theta_j} \geq \f{1}{\sqrt{l}}$. Also, we have $k_j \geq k_*$. Now, we observe that $\Sl_{\abs{\theta_{j}}}^{k_*-1} \subset \Ga \theta$, therefore by taking $a = \abs{\theta_j}$ in \eqref{Lower bound for the minimal sphere}, we have
	\begin{align*}
		M_\Ga \(\theta, \eta\) \geq M \geq C \(\f{\abs{\theta_j}}{\eta}\)^{k_*-1} \geq C \(\f{1}{\sqrt{l} \eta}\)^{k_*-1} \geq C \(\f{1}{\eta}\)^{K_*-1}, \quad \forall \eta < \f{1}{\sqrt{l}}.
	\end{align*}
	We note that the constant $C$ is independent of $\eta$. Hence, $\Ga$ satisfies $(QOP)_{\Ga, d}$, with $d = k_*-1$ and $\eta_0 = \f{1}{\sqrt{l}}$. 
\end{proof}

\begin{remark}
	The same condition holds for the group $\Ga = SO(k_1) \times SO(k_2) \times \cdots \times SO(k_l)$.
\end{remark}

\medskip

\begin{remark}
	If there exists some $\theta \in \Sl^{N-1}$, for which the orbit $\Ga \theta$ is of finite cardinality, then the group $\Ga$ does not satisfy the condition $(QOP)_{\Ga, d}$, for any $d>0$. For example take the action of $O(2) \times O(2) \times \mathbb{Z}_2$ or $O(2) \times O(2) \times Id$ on $\R^5$.
\end{remark}

\section{Appendix}\label{Appendix}

\subsection{Strong continuity of the group action}\label{Appendix1}

\begin{proof}[Proof of \Cref{Representation}]
	It is evident that $T$ is a group homomorphism. Now, to prove $T$ is strongly continuous, let us take a sequence $\{g_n\} \subset \Ga$ such that $g_n \to I$ as $n \to \infty$.\\
	
	\medskip
	
	\textbf{Claim:} For every $u \in H^1\(\bn\), \; g_n \to I \; \Rightarrow T_{g_n} u \to u$ in $H^1\(\bn\)$.\\
	
	Using the facts that $C_c^\infty \(\bn\)$ is dense in $H^1\(\bn\)$ and for every $g \in \Ga$, $T_g$ is an isometry on $H^1\(\bn\)$, it is enough to prove the above convergence when $u \in C_c^\infty(\bn)$. Let us now fix a $C_c^\infty$-function $u$. Since $\phi : \Ga \to \mathbb{Z}_2$ is continuous, $\phi^{-1}\(\{1\}\)$ is an open neighborhood of $I$. Therefore, without loss of generality, we can assume that for every $n$,
	\begin{align}
		\phi\(g_n\) = 1.
	\end{align}
	Consequently, we have
	\begin{align*}
		T_{g_n} u(x) = u \(g_n^{-1}(x)\).
	\end{align*}
	Let us suppose $\mathrm{Supp} (u) \subset B_R = B\(0, \tanh \f{R}{2}\)$. Since the action of $O(N)$ preserves the Euclidean distance, we have
	\begin{align*}
		\mathrm{Supp} \(u \circ g_n^{-1}\) \subset B\(0, \tanh \f{R}{2}\), \; \text{i.e.,} \; \mathrm{Supp} \(u \circ g_n^{-1} - u\) \subset B\(0, \tanh \f{R}{2}\).
	\end{align*}
	Using the fact $g_n^{-1} \to I$, we have
	\begin{align*}
		\sup_{x \in B\(0, \tanh \f{R}{2}\)} \abs{g_n^{-1} x - x} \to 0.
	\end{align*}
	Now for any uniformly continuous function $v$ on $\overline{B\(0, \tanh \f{R}{2}\)}$, the above implies
	\begin{align}
		\sup_{x \in B\(0, \tanh \f{R}{2}\)} \abs{v\(g_n^{-1} x\) - v(x)} \to 0. \label{convergence for uniformly continuous function}
	\end{align}
	Now we have
	\begin{align*}
		\abs{\nabla \(u \circ g_n^{-1}\)(x) - \nabla u(x)} & = \abs{g_n \nabla u \(g_n^{-1} x\) - \nabla u(x)}\\
		& \leq \abs{\(g_n - I\) \nabla u \(g_n^{-1} x\)} + \abs{\nabla u \(g_n^{-1} x\) - \nabla u(x)}
	\end{align*}
	Both the terms uniformly converge to $0$ in $B\(0, \tanh \f{R}{2}\)$, because of $g_n \to I$ and \eqref{convergence for uniformly continuous function} respectively. We have that for any $\la < \la_1$, $\norm{\cdot}_\la$ is an equivalent norm on $H^1\(\bn\)$, therefore
	\begin{align*}
		\norm{T_{g_n} u - u}_\la^2 \leq C \int_{\bn} \abs{\nabla_{\bn} \(T_{g_n} u - u\)}^2 \dv & = C \int_{\Bee} h^{N-2} \abs{\nabla \(T_{g_n} u - u\)}^2 \dx\\
		& \leq C \int_{B\(0, \tanh \f{R}{2}\)} \abs{\nabla \(T_{g_n} u - u\)}^2 \dx\\
		& \leq C \norm{\nabla \(u \circ g_n^{-1}\) - \nabla u}_{L^{\infty} \(B\(0, \tanh \f{R}{2}\)\)}\\
		& \to 0.
	\end{align*}
\end{proof}

\subsection{Local decay and compactness}\label{Decay and compactness}

Again we take $\rn = \R^{k_1} \times \R^{k_2} \times \cdots \times \R^{k_l}, \; \Ga = O(k_1) \times O(k_2) \times \cdots \times O(k_l)$ and
\begin{align*}
	k_\ast = \min_{1 \leq i \leq l} k_i \geq 2.
\end{align*}
Let $x = \(\rho, \theta\) \in \bn \setminus \{0\}$, with $\theta = \(\theta_1, \cdots, \theta_l\)$ and let us define
\begin{align*}
	Z_i(x) = \abs{\theta_i} \sinh \rho, \quad \text{for} \; 1 \leq i \leq l.
\end{align*}

Furthermore, let us take the continuous group homomorphism to be $\phi : \Ga \to \mathbb{Z}_2$, which may not be onto, and we consider the equivariant subspace $H^1\(\bn\)^\phi$. For the functions in this subspace, the orbit geometry can be described separately for each block, yielding a more precise decay estimate than the bound from the $(QOP)_{\Ga, d}$ condition. Since these are arbitrary functions, we do not have the elliptic regularity used in \Cref{Decay for solutions} to obtain pointwise decay. Instead, we obtain a local integral estimate. The estimate below explicitly shows how the different blocks contribute to the decay.\\

\begin{proposition}
	Fix $\zeta > 0$. For every $q \in \[2,2^*\]$, there exists a constant $C \equiv C \(\zeta, q, \Ga\) > 0$ such that for every $u \in H^1\(\bn\)$ for which $|u|$ is $\Ga$-invariant, we have
	\begin{align}
		\int_{B_\zeta(x)} \abs{u}^q \dv \leq C \norm{u}_{H^1\(\bn\)}^q \prod_{i=1}^{l} \(1 + Z_i(x)\)^{-k_i+1} \leq C \norm{u}_{H^1\(\bn\)}^q \(1 + \f{1}{\sqrt{l}} \sinh \rho \)^{-k_\ast + 1}. \label{Local decay estimate}
	\end{align}
	Consequently, for every $2<q<2^*$, the embedding $H^1\(\bn\)^\phi \hookrightarrow L^q\(\bn\)$ is compact.
\end{proposition}

\begin{proof}
	For a fixed $i \in \{1, \cdots, l\}$, the $O(k_i)$-orbit of $\theta_i$ is the sphere $\Sl_{\abs{\theta_i}}^{k_i-1} \subset \R^{k_i}$. Let $\eta > 0$ be arbitrary.\\
	
	\textbf{Claim:} $M_\Ga (\theta, \eta) \geq C \prod_{i=1}^{l} \(1 + \f{\abs{\theta_i}}{\eta}\)^{k_i-1}$.
	
	\medskip
	
	Let $M_i$ be the maximal number of points in $\Sl_{\abs{\theta_i}}^{k_i-1}$, that are $\eta$ separated.\\
	
	\textbf{Case I:} $\abs{\theta_i} \leq \eta$.
	
	\medskip
	
	This is a trivial case, i.e., we can have one point, and then the following holds:
	\begin{align*}
		1 \geq 2^{-(k_i-1)} \(1 + \f{\abs{\theta_i}}{\eta}\)^{k_i-1}.
	\end{align*}
	
	\textbf{Case II:} $\abs{\theta_i} > \eta$.\\
	
	We observe that we can have an inequality similar to \eqref{Lower bound for the minimal sphere}, for every $k_i,\; i \in \{1, \cdots, l\}$, i.e., 
	\begin{align*}
		M_i \geq C_{k_i} \( \f{\abs{\theta_i}}{\eta}\)^{k_i-1} \geq C_{k_i} \(1 + \f{\abs{\theta_i}}{\eta}\)^{k_i-1},
	\end{align*}
	where $C_{k_i}$ is a generic constant, independent of $\eta$. Therefore
	\begin{align*}
		M_{\Ga} \(\theta, \eta\) \geq \prod_{i=1}^{l} M_i \geq C \prod_{i=1}^{l} \(1 + \f{\abs{\theta_i}}{\eta}\)^{k_i-1},
	\end{align*}
	where $C = \prod_{i=1}^{l} C_{k_i}$. This completes the proof of the claim.\\
	
	Now, let us take $\eta = \f{2 \sinh \zeta}{\sinh \rho}$, then we have
	\begin{align*}
		m_\Ga (x, \zeta) \geq M_\Ga (\theta, \eta) \geq C \prod_{i=1}^{l} \(1 + |\theta_i| \f{\sinh \rho}{2 \sinh \zeta}\)^{k_i-1}.
	\end{align*}
	Let us choose $C_\zeta = \min \{1, \f{1}{2\sinh \zeta}\}$, then 
	\begin{align*}
		1 + \f{t}{2\sinh \zeta} \geq C_\zeta (1+t), \quad \forall t \geq 0.
	\end{align*}
	Therefore, we get
	\begin{align*}
		m_{\Ga}(x, \zeta) \geq C \prod_{i=1}^{l} \(1 + \f{Z_i(x)}{2 \sinh \zeta}\)^{k_i - 1} \geq C \prod_{i=1}^{l} \(1 + Z_i(x)\)^{k_i-1}.
	\end{align*}
	Let $m = m_\Ga(x,\zeta)$, then there exists a maximal family of group elements $g_1, \cdots, g_m \in \Ga$ such that 
	\begin{align*}
		B_\zeta(g_ix) \cap B_\zeta(g_jx) = \emptyset, \quad \forall i \neq j, \; \text{and} \; i,j \in \{1, \cdots, m\}.
	\end{align*}
	Since $|u|$ is $\Ga$-invariant, we argue
	\begin{align*}
		\int_{B_\zeta (g_jx)} |u_n|^q \dv = \int_{B_\zeta(x)} |u|^q \dv, \quad \forall j \in \{1, \cdots, m\}.
	\end{align*}
	Therefore, from the Poincar\'e-Sobolev inequality \eqref{Poincare-Sobolev} we have
	\begin{align*}
		m \int_{B_\zeta (x)} \abs{u}^q \dv = \sum_{j=1}^{m} \int_{B_\zeta (g_jx)} \abs{u}^q \dv \leq \int_{\bn} \abs{u}^q \dv \leq C \norm{u}_{H^1\(\bn\)}^q.
	\end{align*}
	We have $\sum_{i=1}^{l} \abs{\theta_i}^2 = 1$ and there exists $j \in \{1, \cdots, l\}$ such that $\abs{\theta_j} \geq \f{1}{\sqrt{l}}$. Then $Z_j(x) \geq \f{1}{\sqrt{l}} \sinh \rho$ and this implies
	\begin{align*}
		\prod_{i=1}^{l} \(1 + Z_i(x)\)^{k_i-1} \geq \(1 + Z_j(x)\)^{k_j-1} \geq \(1 + Z_j(x)\)^{k_\ast-1} \geq \(1 + \f{1}{\sqrt{l}} \sinh \rho \)^{k_\ast-1}.
	\end{align*}
	This proves \eqref{Local decay estimate}.\\
	
	To prove compactness, let $u_n \rightharpoonup 0$ in $H^1\(\bn\)^\phi$. Then $\abs{u_n}$ are $\Ga$-invariant and bounded in $H^1\(\bn\)$. Taking $q=2$ in \eqref{Local decay estimate}, we have
	\begin{align}
		\sup_{x \in B_R^c} \int_{B_\zeta (x)} \abs{u}^q \dv \leq C \(1 + \f{1}{\sqrt{l}} \sinh R \)^{1-k_*}. \label{Exterior estimate for compactness}
	\end{align}
	From the Rellich-Kondrachov compactness theorem, we get
	\begin{align}
		\int_{B_{R+\zeta}} \abs{u_n}^2 \dv \to 0, \quad \text{as} \; n \to \infty. \label{Interior estimate for the compactness}
	\end{align}
	Now combining \eqref{Exterior estimate for compactness}, \eqref{Interior estimate for the compactness} and taking $R \to \infty$ we have
	\begin{align*}
		\sup_{x \in \bn} \int_{B_\zeta(x)} \abs{u_n}^2 \dv \to 0 \quad \text{as} \; n \to \infty.
	\end{align*}
	The standard Lions vanishing lemma (Lemma 1.21 \cite{Willem}) can be adapted to the hyperbolic space $\bn$ and this yields, as $n \to \infty$, $u_n \to 0$ in $L^q\(\bn\)$ for every $2<q<2^*$.
\end{proof}

 \subsection*{Acknowledgment}

The author, Atanu Manna is supported by the Ministry of Education, Govt. of India.

\subsection*{Conflict of Interest statement: } We declare no conflict of interest regarding the publication of this paper.
\subsection*{Data availability statement: } We do not analyse or generate any datasets, because our work proceeds within a theoretical and mathematical approach.


	\bibliographystyle{plain}

\end{document}